\documentclass[11pt]{article}
\usepackage{amsmath,amsthm,amsfonts,amssymb,amscd, amsxtra,color,subfigure}
\usepackage{enumerate}

\usepackage{color}
\usepackage{float}
\usepackage{textcomp}
\usepackage{amsthm}
\usepackage{amsmath}
\usepackage{amssymb}

\usepackage{amsmath,amsthm,amsfonts,amssymb,amscd, amsxtra,color}
\usepackage[active]{srcltx}
\usepackage{cite}
\usepackage{multirow}

\usepackage{amssymb,amsfonts,amstext,amsmath}
\usepackage{changes}

\usepackage{colortbl,hhline,color,soul,url}
\usepackage{changes}
\usepackage{mathtools}

\theoremstyle{plain}

\newtheorem{theorem}{Theorem}
\newtheorem{lemma}{Lemma}
\newtheorem{definition}{Definition}

\newtheorem{proposition}{Proposition}

\newtheorem{example}{Example}
\DeclareMathOperator{\grad}{grad}  

\DeclareMathOperator{\gph}{gph}  
\DeclareMathOperator{\reg}{reg}  
\DeclareMathOperator{\tr}{tr}  
\DeclareMathOperator{\inte}{int}  
\DeclareMathOperator{\dom}{dom} 
\DeclareMathOperator{\lip}{lip} 
\DeclareMathOperator{\id}{id}

\DeclareMathOperator{\dt}{dt}  
\DeclareMathOperator{\de}{d}  
\DeclareMathOperator{\B}{B}

\begin{document}


\title{\textbf{Inexact Newton Methods for Solving Generalized Equations on Riemannian Manifolds}}

\author{Mauricio S. Louzeiro\thanks{School of Computer Science and Technology, Dongguan University of Technology, Dongguan, Guangdong, China and  IME/UFG, Avenida Esperan\c{c}a, s/n, Campus Samambaia, CEP 74690-900, Goi\^ania, GO, Brazil  (Email: {\tt mauriciolouzeiro@ufg.br}). This author was partially supported by National Natural Science Foundation of China (No. 12171087).}
\and
Gilson N. Silva\thanks{Departamento de Matem\'atica, Universidade Federal do Piau\'i­, Teresina, Piau\'i­, 64049-550, Brazil (Email: {\tt
      gilson.silva@ufpi.edu.br}). This author was partially supported  by
     CNPq, Brazil (401864/2022-7 and 306593/2022-0).}
 \and
 Jinyun Yuan\thanks{School of Computer Science and Technology, Dongguan University of Technology, Dongguan, Guangdong, China (Email: {\tt yuanjy@gmail.com}). This author was partially supported by National Natural Science Foundation of China (No. 12171087), Dongguan University of Technology, China (221110093
 ), and Shanghai Municipal Science and Technology Commission, China (23WZ2501400).}
 \and 
Daoping Zhang \thanks{School of Mathematical Sciences and LPMC, Nankai University, Tianjin 300071, China (Email: {\tt daopingzhang@nankai.edu.cn}). This author was supported by National Natural Science Foundation of China (No. 12201320) and the Fundamental Research Funds for the Central Universities, Nankai University (No. 63221039 and 63231144).} 
}

\maketitle

\begin{abstract}
{The convergence of inexact Newton methods is studied for solving generalized equations on Riemannian manifolds by using the metric regularity property, which is also explored. Under appropriate conditions and without any additional geometric assumptions, local convergence results with linear and quadratic rates, as well as a semi-local convergence result, are obtained for the proposed method. Finally, the theory is applied to the problem of finding a singularity for the sum of two vector fields. In particular, the KKT system for the constrained Riemannian center of mass on the sphere is explored numerically.\\}


\end{abstract}

\maketitle
\section{Introduction}\label{sec:int}


In recent years, constrained generalized equations on Banach spaces, i.e., finding a solution to the inclusion
\begin{equation}\label{eq:CGEBS}
x\in C, \quad f(x) + F(x) \ni 0
\end{equation}
where $\mathbb{X}$ and $\mathbb{Y}$ are two Banach spaces, $C\subset \mathbb{X}$ is a nonempty, closed, and convex set, $f\colon \mathbb{X} \to \mathbb{Y}$ is a mapping, and $F \colon \mathbb{X} \rightrightarrows \mathbb{Y}$ is a set-valued mapping, have gained increased attention \cite{Oliveira2019, andreani2022}. This is attributed to the fact that the model \eqref{eq:CGEBS} covers many well-known problems, such as constrained variational inequality and split variational inequality problem \cite{Censor2011, He2015}, nonlinear equations, systems of equations and inequalities, optimality condition in mathematical programming and optimal control, and equilibrium problem. The readers are referred to \cite{Adly2015, Adly2022, Adly2018, izmailov2010inexact, Dontchev1996, CibulkaDontchev2015, FerreiraSilva2017, ferreira2017metrically, FerreiraSilva2018, DontchevAragon2011, DontchevRockafellar2014, DontchevRockafellar2010, DontchevRockafellar2013, DontchevAragon2014, GEOFFROY2004, GayduSilva2020, KlatteKummer} for a detailed study of \eqref{eq:CGEBS} with $C= \mathbb{X}$.

For a Riemannian manifold ${\cal M}$, a closed set $\Omega \subset {\cal M}$ with a nonempty interior, a continuously differentiable mapping $f \colon {\cal M} \to \mathbb{R}^m$, and a set-valued mapping $F \colon {\cal M} \rightrightarrows \mathbb{R}^m$, we consider the generalized equation
\begin{equation}\label{eq:mainproblem}
p \in \Omega, \quad f(p) + F(p) \ni 0.
\end{equation}
Evidently, \eqref{eq:mainproblem} covers the nonlinear equation $f(p) = 0$ ($F \equiv 0)$ and the nonlinear inclusion problem $f(p)\in K$ ($F \equiv -K$) for a fixed cone $K\subset\mathbb{R}^m$
  \cite{ChongLi2009}. In this paper, we shall prove that problem \eqref{eq:mainproblem} also covers the problem of finding a singularity of the sum of two vector fields of the form
\begin{equation}\label{eq:VF}
p \in \Omega, \quad V(p) + Z(p) \ni 0_p,
\end{equation}
where $V\colon {\cal M}\to {\cal T}{\cal M}$ is a single-valued vector field, 
$Z\colon {\cal M} \rightrightarrows {\cal T}{\cal M}$ is a set-valued vector field. In particular, we demonstrate that \eqref{eq:mainproblem} can also be used to obtain a solution for a variational inequality problem and the Karush–Kuhn–Tucker (KKT) conditions for a constrained optimization problem on manifolds.


Problem \eqref{eq:VF} has been extensively investigated \cite{Adly2022, ferreira2017metrically} and solved using Newton's method for \( Z \equiv 0 \) \cite{alvarez2008unifying, fernandes2017superlinear, Udriste1994, li2005convergence, adler2002newton, dedieu2003newton}. This problem naturally arises, for example, in the first-order optimality conditions of the minimization problem
\begin{equation} \label{pro:sf+g}
\text{Minimize } \varsigma (p) + \vartheta (p), \quad p \in \inte \Omega,
\end{equation}
where \(\varsigma \colon {\cal M} \to \overline{\mathbb{R}} \coloneqq \mathbb{R} \cup \{\pm\infty\}\) is a differentiable function defined over \(\inte \Omega\) (the interior of \(\Omega\)) and \(\vartheta \colon {\cal M} \to \overline{\mathbb{R}}\) is non-differentiable. In fact, one can consider \eqref{eq:VF} with \(V\) representing the Riemannian gradient of \(\varsigma\) (\(\grad \varsigma\)), \(Z\) representing the Riemannian subdifferential of \(\vartheta\) (\(\partial \vartheta\)), and \(\Omega = \inte \Omega\), i.e.,
\[
p \in \inte \Omega, \quad \grad \varsigma(p) + \partial \vartheta(p) \ni 0_p.
\]
Problem \eqref{pro:sf+g} is associated with several important applications, including sparse principal component analysis \cite{genicot2015weakly}, sparse blind deconvolution \cite{zhang2017global}, unsupervised feature selection \cite{tang2012unsupervised}, and image restoration \cite{bergmann2016parallel, bergmann2021fenchel}.


The extensive scope of generalized equations and the growing interest in optimization on manifolds \cite{sato2021riemannian, boumal2023intromanifolds, absil2009optimization} in recent years have motivated us to explore the integration of these two areas in this paper. Our focus is on studying a Riemannian version of the inexact Newton method proposed by \cite{Dontchev2015} with the aim of solving problem \eqref{eq:mainproblem}. Given an initial point \( p_0 \in \mathcal{M} \), our method generates a sequence of iterations \( \{p_k\} \) as follows:
\begin{equation}\label{eq:INMtp}
\left( f(p_k) + \mathcal{D}f(p_k)[v_k] + F(\exp_{p_k} v_k) \right) \cap R_k(p_k) \neq \varnothing, \qquad p_{k+1} \coloneqq \exp_{p_k} v_k,
\end{equation}
where \( R_k \colon \mathcal{M} \rightrightarrows \mathbb{R}^m \) is a sequence of set-valued mappings with closed graphs representing the inexactness, and \( \mathcal{D}f \) denotes the differential of \( f \). In other words, the method involves: selecting an initial point \( p_0 \) on the manifold (sufficiently close to a solution), solving a subproblem to find \( v_k \) in the intersection above (under certain conditions, it is possible to guarantee that it is nonempty), and computing the next iteration \( p_{k+1} \) by applying the exponential map to \( (p_k, v_k) \). It is noteworthy that whenever \( v_k \) is sufficiently small, \eqref{eq:INMtp} can be expressed as:
\begin{equation}\label{eq:INM}
\left( f(p_k) + \mathcal{D}f(p_k)[ \exp^{-1}_{p_k}p_{k+1} ] + F( p_{k+1} ) \right) \cap R_k(p_k) \neq \varnothing.
\end{equation}
Although it may appear that obtaining \( v_{k} \) requires solving the subproblem in \eqref{eq:INMtp} exactly, the map \( R_k \) is specifically introduced to circumvent this necessity. This is evident in the particular case of \eqref{eq:INMtp} where \( R_k(p_k) \) is defined as the closed ball centered at \( 0 \) with a radius of \( \eta_k\|f(p_k)\|_{e} \), with \( \eta_k \) being a forcing sequence of non-negative real numbers converging to 0, and \( F \equiv 0 \). In this scenario, the subproblem in \eqref{eq:INMtp} consists of finding a \( v_k \) that satisfies
\[
\|f(p_k) + {\cal D}f(p_k)[v_k]\|_{e} \leq \eta_k\|f(p_k)\|_{e}, \qquad k = 0, 1, \ldots,
\]
where \(\|\cdot\|_e\) denotes the Euclidean norm, which can be interpreted as an inexact Newton method for by solving \( f(p) = 0 \). The convergence analysis presented in this paper is conducted using \eqref{eq:INMtp} because it is more general and can be applied to other potential particular cases of \eqref{eq:mainproblem}.

Assuming that the set-valued mapping \( f+F \) in \eqref{eq:mainproblem} is metrically regular at \(\bar{p}\in \Omega\) for \(0\) and that \( R_k \) satisfies a suitable boundedness condition, we demonstrate that a sequence \(\{p_k\}\) generated by \eqref{eq:INM} exhibits both linear and quadratic convergence towards \(\bar{p}\), with the exact nature of the convergence depending on the specific assumptions regarding \( R_k \). These results can be obtained without requiring prior knowledge of the sequence of mappings \( R_k \) in terms of problem-specific data. However, it is necessary to ensure that a sequence \( \{u_k\} \) in \( R_k(p_k) \) converges to \(0\) at the same rate as the sequence \(\{p_k\}\) converges to \(\bar{p}\). This requirement is a standard assumption in the context of inexact Newton-type methods, even when applied to nonlinear equations.

Under the condition of metric regularity of a linearization of \( f+F \) at \(\bar{p}\) for \(0\), and provided that certain additional conditions are met, we present variations of the aforementioned results. In these variations, a neighborhood of \(\bar{p}\) is assumed to be known, which allows for a more suitable choice of the initial point \( p_0 \) for the sequence \(\{p_k\}\). We also provide a semi-local convergence result, where the required conditions are related to \( p_0 \) rather than \(\bar{p}\). This result is new even in the case where \({\cal M}\) is a Euclidean space.

To \textcolor{red}{} understand the concept of metric regularity well on Riemannian manifolds, we constructed examples of mappings that are metrically regular over the set of positive definite symmetric matrices equipped with a well-established Riemannian metric. Additionally, we present conditions that guarantee the metric regularity property for certain mappings, and in particular, we compare the metric regularity of \( f+F \) with that of its linearization. Finally, some examples are given to show that our proposed concept is useful and  a numerical example applied to the KKT system is given as well for the constrained Riemannian center of mass on the sphere to illustrate our theoretical results.

This work is organized as follows. In Section~\ref{sec:int}, some notations and basic concepts are reviewed. In Section~\ref{sec:MR} the metric regularity assumption is explored. In Section~\ref{sec:conv}, local and semi-local convergence are studied for the proposed method. In Section~\ref{sec:rpscandeg} the relationship is investigated between \eqref{eq:mainproblem} and \eqref{eq:VF}, and some examples of classical  problems which can be viewed as generalized equations are given. In Section~\ref{sec:numerical_example}, a numerical example is provided to illustrate our theoretical results. Finally, conclusions are presented in the last section.


\section{Preliminary} \label{sec:int}

In this section we recall some notations, definitions and basic properties of Riemannian manifolds used throughout the paper, which can be found in many introductory books on Riemannian geometry, for example \cite{DoCa92, lee2006riemannian,Lee:2003:1, Tu:2011:1}.


Suppose that ${\cal M}$ is a connected, $n$-dimensional smooth manifold. At each point $p\in {\cal M}$, the tangent space ${\cal T}_p{\cal M}$ is an $n$-dimensional vector space with its origin at $0_p$. The disjoint union of all tangent spaces, denoted as ${\cal T}{\cal M}$, is the {\it tangent bundle} of ${\cal M}$. We assume that ${\cal M}$ is equipped with a Riemannian metric, making it a {\it Riemannian manifold}. At each point $p\in {\cal M}$, this Riemannian metric is denoted by $ {\langle} \cdot , \cdot {\rangle}_p\colon {\cal T}_p{\cal M} \times {\cal T}_p{\cal M} \to \mathbb{R}$, and the associated norm in ${\cal T}_p{\cal M}$ is represented as $\| \cdot \|_p$.
The Riemannian distance between two points $p$ and $q$ in ${\cal M}$, denoted as $d(p,q)$, is defined as the infimum of the lengths of all piecewise smooth curve segments connecting $p$ to $q$. Additionally, the distance from a point $p$ to a subset ${\cal W} \subset {\cal M}$  is defined as $d( p , {\cal W} )\coloneqq \inf_{q\in {\cal W} }d(p,q)$ and the interior of ${\cal W}$ is represented by  $\inte {\cal W}$.

A {\it vector field} $V$ on ${\cal M}$ is a correspondence that associates to each point $p\in {\cal M}$ a vector $V(p)\in {\cal T}_p {\cal M}$. The point $p$ is said to be a singularity of $V$ if and only if $V(p)=0_p$. The set of smooth vector fields on ${\cal W} \subseteq {\cal M}$ is denoted by ${\cal X}({\cal W})$.

The tangent vector of a smooth curve $\gamma\colon I \to {\cal M}$ defined on some open interval $I \subseteq \mathbb{R}$ is denoted by $\dot{\gamma}(t)$. For each $a, t \in I$, $a < t$, the Levi-Civita connection $\nabla\colon {\cal X}({\cal M}) \times {\cal X}({\cal M}) \to {\cal X}({\cal M})$ induces an isometry $P_{\gamma, a, t}\colon {\cal T}_{\gamma(a)} {\mathcal{M}} \to {\cal T}_{\gamma(t)} {\mathcal{M}}$ relative to the Riemannian metric on ${\cal M}$, given by $P_{\gamma, a, t}, v = V(\gamma(t))$, where $V$ is the unique vector field on $\gamma$ such that $\nabla_{\dot{\gamma}(t)}V(\gamma(t)) = 0$ and $V(\gamma(a)) = v$. The isometry $P_{\gamma, a, t}$ is the {\it parallel transport} along $\gamma$ joining $\gamma(a)$ to $\gamma(t)$. When the geodesic $\gamma$ connecting $p = \gamma(a)$ and $q = \gamma(t)$ is unique, the notation $P_{pq}$ will be used instead of $P_{\gamma, a, t}$. It is well-known that $P_{qp} \circ P_{pq}$ is equal to the identity map over ${\cal T}_p{\cal M}$.

The {\it differential} of a smooth function $f\colon {\cal M} \to \mathbb{R}$ at $p$ is the linear map ${\cal D}f(p)\colon  {\cal T}_p{\cal M} \to \mathbb{R}$ which assigns to each $v\in {\cal T}_p{\cal M}$ the value
$$
{\cal D}f(p)[v] = \dot{\gamma}(t_0)[f] = \frac{\de}{\dt}(f\circ \gamma)\Bigl|_{t=t_0},
$$
for every smooth curve $\gamma\colon I  \to {\cal M}$ satisfying $\gamma(t_0)=p$ and $\dot{\gamma}(t_0)=v$. The {\it gradient} at $p$ of $f$, denoted as $\grad f(p)$, is defined by the unique tangent vector at $p$ such that 
$
{\langle}  \grad f(p)  ,   v  {\rangle}_p = {\cal D}f(p)[v]$ for all $v\in {\cal T}_p{\cal M}.
$
For a smooth multifunction $f\coloneqq (f_1,\ldots,f_m) \colon {\cal M} \to \mathbb{R}^m$, its differential ${\cal D}f(p)\colon  {\cal T}_p{\cal M} \to \mathbb{R}^m$ is given by 
\begin{equation}\label{eq:diff.grad. multi}
{\cal D}f(p)[v]= ({\langle}  \grad f_1(p)  ,   v  {\rangle}_p, \ldots, {\langle}  \grad f_m(p)  ,   v  {\rangle}_p), \qquad  v\in {\cal T}_p{\cal M}.
\end{equation}
Note that ${\cal D}f(p)$ is a linear map from ${\cal T}_p{\cal M}$ into $\mathbb{R}^m$ for all $p\in {\cal M}$. The norm of a linear map $ A \colon  {\cal T}_p{\cal M} \to \mathbb{R}^m$ is defined  by $\|A\|_{map} \coloneqq \sup\{  \| A v \|_{e}\colon  v\in {\cal T}_p{\cal M}, \, \|v\|_p=1  \}$ where $\|\cdot\|_e$ is the Euclidean norm on $\mathbb{R}^m$. 

A vector field $V$ along a smooth curve $\gamma$ is said to be {\it parallel} if and only if $\nabla_{ \dot{\gamma} } V=0$. The curve $\gamma$ is a {\it geodesic} when $ \dot{\gamma}$ is self-parallel. When the geodesic equation $\nabla_{\dot{\gamma}} \dot{\gamma}=0$ is a second-order nonlinear ordinary differential equation, the geodesic $\gamma=\gamma_{v}( \cdot ,p)$ is determined by its position $p$ and velocity $v$ at $p$. The restriction of a geodesic to a closed bounded interval is called a {\it geodesic segment}. Denote the unique geodesic segment $\gamma \colon [0,1] \to {\cal M}$ satisfying $\gamma(0)=p$ and $\gamma(1)=q$ by $\gamma_{pq}$. A geodesic segment joining $p$ to $q$ in $\mathcal{M}$ is said to be {\it minimal} if its length is equal to $d(p,q)$. A Riemannian manifold is {\it complete} if its geodesics are defined for all values of $t\in \mathbb{R}$. Hopf-Rinow's theorem asserts that every pair of points in a complete, connected Riemannian manifold $\mathcal{M}$ can be joined by a (not necessarily unique) minimal geodesic segment. Due to the completeness of $\mathcal{M}$, the {\it exponential map} $\exp_{p}\colon {\cal T}_{p} \mathcal{M} \to \mathcal{M} $ is given by $\exp_{p}v = \gamma _{v}(1,p)$, for each $p\in \mathcal{M}$. {\it In this paper, all manifolds are assumed to be connected, finite-dimensional, and complete.}

The open and closed balls on ${\cal M}$ of radius $r>0$ centered at $p \in {\cal M}$ are defined, respectively,
as 
$
{\cal B}_{r}(p)\coloneqq \left\{q\in {\cal M} \colon d(p,q) < r \right\}
$ and $
{\cal B}_{r}[p]\coloneqq \left\{q\in {\cal M} \colon d(p,q) \leq r \right\}
$. Analogously, the open and closed balls on ${\cal T}_p{\cal M}$ of radius $r > 0$ centered at $u\in {\cal T}_p{\cal M}$ are  defined, respectively, as $B_{r}(u)\coloneqq \left\{v\in {\cal T}_p{\cal M} \colon \| v - u\|_p < r \right\}$ and $B_{r}[u]\coloneqq \left\{v\in {\cal T}_p{\cal M} \colon \| v - u\|_p \leq r \right\}$.  It is well-known that there exists $r>0$ such that $\exp_p\colon B_r(0_p) \to {\cal B}_r(p)$ is a diffeomorphism, and ${\cal B}_r(p)$ is called a {\it normal ball} 
with center $p$ and radius $r$. Whenever ${\cal B}_r(p)$ is a normal ball, it and its closure will be denoted by $\B_r(p)$ and $\B_r[p]$, respectively. The {\it injectivity radius of $\cal M$ at $p$}, denoted by $r_{inj}(p)$,  is the supremum of all $r>0$ such that ${\cal B}_r(p)$ is a normal ball. The equality
 \begin{equation}\label{eq:dist.iq.nor.expi}
	d(q,p)=\|\exp_{p}^{-1}q\|_p
\end{equation}
holds for all $q\in {\cal B}_r(p)$, where $\exp_{p}^{-1}\colon {\cal B}_r(p)   \to {\cal T}_p{\cal M}$ denotes the inverse of the exponential map.
Recall that there exist $r>0$ and $\delta>0$ such that, for every $q \in \B_{r}(p)$, ${\cal B}_{\delta}(q)$ is a normal ball and $\B_{r}(p) \subset {\cal B}_{\delta}(q)$, see \cite[Theorem 3.7]{DoCa92}. In this case, $\B_{r}(p)$ is called a {\it totally normal ball} of center $p$ and radius $r$. When \(\mathcal{M}\) is a {\it Hadamard manifold}, that is, a complete, simply connected Riemannian manifold with nonpositive sectional curvature, \(\B_r(p)\) is totally normal for all \(r > 0\) and \(p \in \mathcal{M}\).

A sequence $\{p_k\}\subset {\cal M}$ is linearly convergent to $\bar{p}$ when there exist $\theta \in (0,1)$ and $k_0\in \mathbb{N}$ such that
$$
d(p_{k+1}, \bar{p}) \leq \theta d(p_{k}, \bar{p}), \quad \mbox{for all } k>k_0.
$$
It is said to be quadratically convergent to $\bar{p}$ when there exist $\theta > 0$ and $k_0\in \mathbb{N}$ such that
$$
d(p_{k+1}, \bar{p}) \leq \theta d^2(p_{k}, \bar{p}), \quad \mbox{for all } k>k_0.
$$
We end this section by presenting the concept of Lipschitz continuity for the differential of continuously differentiable functions in the Riemannian context. This concept will be fundamental for obtaining the quadratic convergence rate for our algorithm.
\begin{definition} \label{Def:DLips}
	Let $f\colon {\cal M}  \to \mathbb{R}^m$ be continuously differentiable at $\bar{p}$. Then ${\cal D} f$ is
	L-Lipschitz continuous around $\bar{p}$ if there exists $\delta_L>0$ such that for every $p,q \in {\cal B}_{\delta_L}(\bar p)$ 
	\begin{equation}\label{eq:Df-lipschitz}
		\|	{\cal D}f(q) P_{\gamma,0,1}   -  {\cal D}f(p) \|_{map} \leq L \| \dot{\gamma}(0) \|_{p},
	\end{equation}
	where $\gamma \colon [0,1] \to {\cal M}$ is a geodesic connecting $p=\gamma(0)$ to $q=\gamma(1)$.
\end{definition}


\section{Metric Regularity on Riemannian Manifolds}\label{sec:MR}

Let  $F \colon {\cal M} \rightrightarrows \mathbb{R}^m$ be a set-valued mapping, whose the graph of $F$ is the set
$\gph F \coloneqq  \{(p,x)\in {\cal M}  \times \mathbb{R}^m \colon x\in F(p) \}$ with domain $\dom F \coloneqq  \{p\in {\cal M}  \colon F(p) \neq \emptyset \}$. The inverse of $F$ is defined as $x \mapsto F^{-1}(x)=\{p\in {\cal M} \colon x\in F(p)\}$.

A function $f \colon {\cal M}  \to \mathbb{R}^m$ is said to be {\it Lipschitz continuous} relative to a set  ${\cal W} \subset  \dom f$ if there exists a constant $\kappa \geq 0$ such that 
\begin{equation*}
	|f(p') - f(p)| \leq \kappa d(p',p) \mbox{  for all } p',p\in {\cal W}. 
\end{equation*}
Let $\bar{p}\in\inte \dom f$. The {\it Lipschitz modulus} of $f$ at $\bar{p}$ is defined by
\begin{equation*}\label{eq:def.lip}
	\lip(f;\bar{p}) \coloneqq \underset{\substack{ p',p  \to \bar{p}\\p\neq p'}}{\limsup }  \frac{| f(p') - f(p) |}{d(p',p)}.
	\end{equation*}
	We now introduce the following notations on $\mathbb{R}^m$: $\mathbb{B}_{r}[x]$ denotes the Euclidean closed ball of radius $r>0$ and center $x$, and $d_e$ denotes the Euclidean distance from a point to a set. The next concept plays an important role in role in this paper and it comes from \cite[p. 279]{DontchevRockafellar2013} in the metric spaces setting:
	\begin{definition}\label{de:majcon}
	A set-valued mapping $\Phi\colon {\cal M}  \rightrightarrows \mathbb{R}^m$ is said to be metrically regular  at $\bar p\in {\cal M} $ for $\bar x \in \mathbb{R}^m$ when $\bar x\in \Phi(\bar p)$, and there exist positive constants $\sigma$, $a$ and $b$ such that
	\begin{equation}\label{eq:loc.clo.me.re.def}
		\mbox{the set} \,\, \gph \Phi \, \cap ( {\cal B}_{a}[\bar p] \times \mathbb{B}_{b}[\bar x] ) \,\, \mbox{is closed}
	\end{equation}
and
	\begin{equation} \label{def:MR}
		d(p, \Phi^{-1}(x))\leq \sigma d_e(x, \Phi(p))  \,\,\mbox{ for all } \,\,\, (p,x)\in {\cal B}_{a}[\bar p] \times \mathbb{B}_{b}[\bar x].
	\end{equation}
	 The infimum of $\sigma$ over all such combinations of $\sigma$, $a$ and $b$ is called the regularity modulus for $\Phi$ at $\bar p$ for $\bar x$ and denoted by $\reg(\Phi; \bar p| \bar x)$. The absence of metric regularity is signaled by $\reg(\Phi; \bar p| \bar x)=\infty$. 
\end{definition}

It is important to emphasize that one of the main assumptions used to guarantee the convergence of the proposed algorithm in this paper is metric regularity. This concept, in the context of Riemannian manifolds, has been explored in recent works such as \cite{Adly2018,adler2002newton,ferreira2017metrically}. We will now briefly discuss this property in the case where \(\mathcal{M} \equiv \mathbb{R}^n\). If \(f \colon \mathbb{R}^n \to \mathbb{R}^n\) is a differentiable function, the metric regularity of \(f\) at \(\bar{p}\) for \(0\) is equivalent to stating that the inverse of the Jacobian matrix of \(f\) at \(\bar{p}\) exists, which is the standard regularity assumption applied to solve the nonlinear equation \(f(p) = 0\). On the other hand, if we are interested in solving a nonlinear system of equations and inequalities, for example,
\begin{equation}\label{eq:feasiblepoint:1}
g(p) \leq x, \quad h(p) = y,
\end{equation}
where \(x \in \mathbb{R}^m\) and \(y \in \mathbb{R}^s\) are given parameters, and \(g: \mathbb{R}^n \to \mathbb{R}^m\) and \(h: \mathbb{R}^n \to \mathbb{R}^s\) are continuously differentiable functions, then we can associate problem \eqref{eq:feasiblepoint:1} with the following generalized equation:
$$
\xi \in f(p) + F(p), \quad  
\xi = \begin{bmatrix}
x \\
y
\end{bmatrix}, \quad 
f = \begin{bmatrix}
g \\
h 
\end{bmatrix},
\quad 
F = \begin{bmatrix}
\mathbb{R}_+^m \\
0
\end{bmatrix}.
$$
In this particular case, it is well known that the metric regularity of \( f + F\) at \(\bar{p}\) for \(0\) is equivalent to the standard Mangasarian-Fromovitz constraint qualification at \(\bar{p}\). Notably, the case where \(F \equiv N_C\) is of particular interest, with \(N_C\) denoting the normal cone to a closed convex set \(C \subset \mathbb{R}^m\). The metric regularity of \( f + N_C\) at \(\bar{p}\) for \(0\) is equivalent to the concept of \textit{strong metric regularity}, meaning that \(f + N_C\) is locally single-valued and Lipschitz continuous at \(\bar{p}\) for \(0\). Thus, if \(f\) and \(C\) are chosen such that \(f + N_C \ni 0\) forms the KKT system for a constrained optimization problem, then strong metric regularity is equivalent to the linear independence of the gradients of the active constraints and the strong second-order sufficient condition. These details are thoroughly discussed in \cite{DontchevRockafellar2014}; see also Examples 6 and 7 in Section \ref{sec:rpscandeg}.

We present in Subsection \ref{Sec:ExampleSDP} some examples of set-valued mapping that satisfy the previous concept.
The next lemma is a version in the Riemannian setting of Corollary 3F.3 given in \cite{DontchevRockafellar2014}. Its proof is similar to the corresponding result in Euclidean space and will be omitted here.

\begin{lemma}\label{lem:mrsslmrgiz}
Consider a mapping $ \Phi  \colon {\cal M}  \rightrightarrows \mathbb{R}^m$ and a point $(\bar{p}, \bar{x} ) \in \gph \Phi$. Then for every $g \colon {\cal M}  \to \mathbb{R}^m$ with $\lip(g;\bar{p})=0$, one has 
$\reg(g + \Phi; \bar{p}| g(\bar{p}) + \bar{x})=\reg( \Phi  ; \bar{p}|\bar{x})$.
\end{lemma}


In the theorem below, we demonstrate that the metric regularity condition of a mapping at $\bar p\in {\cal M} $ for $\bar x \in \mathbb{R}^m$ is equivalent to the same condition for its linearization at $\bar{p}$.

 \begin{theorem}\label{theo:lyus-graves-coro}
 Consider a normal ball $\B_r(\bar{p}) \subset {\cal M}$. Let  $f \colon {\cal M}  \to \mathbb{R}^m$ be a 
 	continuously differentiable function at $\bar{p}$ and $F \colon {\cal M}  \rightrightarrows \mathbb{R}^m$ be a set-valued mapping. 
 	 Then for  $G\colon {\cal M} \rightrightarrows \mathbb{R}^m$ defined by 
		\begin{equation}\label{eq:mappauxG}
	G(p) =
	\begin{cases}
	 f(\bar{p})+{\cal D}f(\bar p)[\exp^{-1}_{\bar{p}} p] + F(p), & p \in \B_r(\bar{p}),
		\\
		f(p) + F(p), & p\in {\cal M} \backslash \B_r(\bar{p}),
	\end{cases}
\end{equation} 
one has $  \reg(G; \bar p| 0)  =  \reg( f + F ; \bar p| 0 )$.
 \end{theorem}
 \begin{proof}
 	Pick $\epsilon > 0$ arbitrary. By Lemma~\ref{lem:assumptionproof} given in  the Appendix of this paper, there exists a totally normal ball $\B_{\delta_{\epsilon}}(\bar{p}) \subset \B_r(\bar{p})$ such that
 	\begin{equation}\label{eq:assump:lipsh.cond.Theolrm} 
 		\|{\cal D}f(p) [\exp_{p}^{-1}p' - \exp_{p}^{-1}p''] -  {\cal D}f(\bar{p})[ \exp_{\bar{p}}^{-1}p' - \exp_{\bar{p}}^{-1}p'']\|_e 
 		\leq \epsilon d(p',p''), \qquad p, p',p''\in \B_{\delta_{\epsilon}}(\bar{p}).
 	\end{equation}
	Consider a function $g \colon {\cal M} \to  \mathbb{R}^m$ given by
	\begin{equation}\label{eqref:Gfun}
	g(q) =
	\begin{cases}
	 f(\bar{p})-f(q)+{\cal D}f(\bar p)[\exp^{-1}_{\bar{p}} q] , & q \in \B_{r}(\bar{p}),
		\\
		\qquad 0, & q\in {\cal M} \backslash \B_{r}(\bar{p}).
	\end{cases}
\end{equation}
 In particular, this function satisfies
 \begin{equation}\label{eq:prthemrssl.dlg}
 \|g(p'') - g(p')\|_e = \|f(p') - f(p'') - {\cal D}f(\bar p)[ \exp^{-1}_{\bar{p}} p' - \exp^{-1}_{\bar{p}} p''] \|_e, \qquad p',p'' \in \B_{\delta_{\epsilon}}(\bar{p}).
 \end{equation}
The first part of Proposition~\ref{lem:properepsil}  guarantees that there exists $\delta \in (0, \delta_{\epsilon})$ such that 
 $$
f(p') - f(p'')  =   \int_{0}^{1}	{\cal D}f(\gamma_{p''p'}(t)) [\dot{\gamma}_{p''p'}(t)] \dt ,    \qquad p',p''  \in \B_{\delta}(\bar{p}),
 $$
 where $\gamma_{p''p'} \colon [0,1] \to {\cal M}$ is the geodesic satisfying $\gamma_{p''p'}(0)=p''$ and $\gamma_{p''p'}(1)=p'$. Adding $-{\cal D}f(\bar{p})[  \exp^{-1}_{\bar p} p'  -   \exp^{-1}_{\bar p} p''  ]$ to both sides, it follows that 
 \begin{multline*}
 	f(p') - f(p'') - {\cal D}f(\bar p)[ \exp^{-1}_{\bar{p}} p' - \exp^{-1}_{\bar{p}} p'']  =  \\ \int_{0}^{1}	{\cal D}f(\gamma_{p''p'}(t)) [\dot{\gamma}_{p''p'}(t)]  - {\cal D}f(\bar p)[ \exp^{-1}_{\bar{p}} p' - \exp^{-1}_{\bar{p}} p''] \dt ,    \qquad p',p''  \in \B_{\delta}(\bar{p}).
 \end{multline*}
Applying norm on both sides, we conclude that
  \begin{multline}\label{eq:ptmrssli.ineq}
 \|  f(p') - f(p'') - {\cal D}f(\bar p)[ \exp^{-1}_{\bar{p}} p' - \exp^{-1}_{\bar{p}} p'']  \|_{e}  \leq   \\ \int_{0}^{1}\| {\cal D}f(\gamma_{p''p'}(t)) [\dot{\gamma}_{p''p'}(t)]  - {\cal D}f(\bar p)[ \exp^{-1}_{\bar{p}} p' - \exp^{-1}_{\bar{p}} p''] \|_e \dt ,    \qquad p',p''  \in \B_{\delta}(\bar{p}).
 \end{multline}	
On the other hand, from Proposition~\ref{lem:geo.mpoit.exp} with $p=p''$ and $q=p'$ we have
$$
\dot{\gamma}_{p''p'}(t) = \exp_{\gamma_{p''p'}(t)}^{-1}p' -  \exp_{\gamma_{p''p'}(t)}^{-1}p'', \qquad t\in [0,1].
$$
From this equality and  \eqref{eq:assump:lipsh.cond.Theolrm}  with $p=\gamma_{p''p'}(t)$,  \eqref{eq:ptmrssli.ineq} yields
  \begin{equation*}
	\|f(p') - f(p'') - {\cal D}f(\bar p)( \exp^{-1}_{\bar{p}} p' - \exp^{-1}_{\bar{p}} p'') \|_e  \leq \epsilon  d(p',p'') ,    \qquad p',p''  \in \B_{\delta}(\bar{p}).
\end{equation*}
Then \eqref{eq:prthemrssl.dlg} implies that $\| g(p'') - g(p') \|_e  \leq \epsilon  d(p',p'')$ for all  $p',p''  \in \B_{\delta}(\bar{p})$. Since $\epsilon>0$ was chosen arbitrarily,   we can conclude that $\lip(g;\bar{p})=0$. Finally, to obtain the desired equality, just use Lemma~ \ref{lem:mrsslmrgiz} with $g$ as in \eqref{eqref:Gfun} and $\Phi=f+F$.
 \end{proof}

  Next, we define the Aubin property for a set-valued mapping \( \aleph \colon \mathbb{R}^m \rightrightarrows \mathcal{M}\). To this end, it is necessary to introduce the concept of excess between subsets of \(\mathcal{M}\). For sets \({\cal W}_1\) and \({\cal W}_2\) in \(\mathcal{M}\), the \textit{excess of \({\cal W}_1\) beyond \({\cal W}_2\)} is defined as
\[ e({\cal W}_1, {\cal W}_2) = \sup_{p \in {\cal W}_1} d(p, {\cal W}_2), \]
with the convention that \(e(\emptyset, {\cal W}_2) = 0\) for \({\cal W}_2 \neq \emptyset\) and \(e(\emptyset, {\cal W}_2) = \infty\) otherwise.

\begin{definition}\label{def:Aubin_property}
A set-valued mapping \( \aleph \colon \mathbb{R}^m \rightrightarrows \mathcal{M} \) is said to have the Aubin property at \(\bar{x} \in \mathbb{R}^m\) for \(\bar{p} \in \mathcal{M}\) if \(\bar{p} \in \aleph(\bar{x})\), and there exist positive constants \(\sigma\), \(a\), and \(b\) such that
\[
e\left( \aleph(x) \cap \mathcal{B}_{a}[\bar{p}], \aleph(x') \right) \le \sigma \|x - x'\| \quad \text{for all } x, x' \in \mathbb{B}_{b}[\bar{x}].
\]
The infimum of \(\sigma\) over all such combinations of \(\sigma\), \(a\), and \(b\) is called the Lipschitz modulus of \(\aleph\) at \(\bar{x}\) for \(\bar{p}\) and is denoted by \(\operatorname{lip}(\aleph; \bar{x}|\bar{p})\). The absence of this property is indicated by \(\operatorname{lip}(\aleph; \bar{x}|\bar{p}) = \infty\).
\end{definition}
The next result provides a relationship between the metric regularity of a set-valued mapping $\Phi \colon {\cal M} \rightrightarrows \mathbb{R}^m$ and the Aubin property for its inverse. Its proof follows from a straightforward adaptation of the demonstration in \cite[Theorem 3E.6]{DontchevRockafellar2014} and the concepts introduced above, and hence will be omitted.

\begin{theorem}\label{teo:Aubin_MR}
A set-valued mapping $\Phi \colon {\cal M} \rightrightarrows \mathbb{R}^m$ is metrically regular at $\bar{p}$ for $\bar{x}$ with a constant $\sigma$ if and only if its inverse $\Phi^{-1} : \mathbb{R}^m \rightrightarrows {\cal M}$ has the Aubin property at $\bar{x}$ for $\bar{p}$ with constant $\sigma$.
Thus,
$
\operatorname{lip}(\Phi^{-1}; \bar{x}|\bar{p}) = \operatorname{reg}(\Phi; \bar{p}|\bar{x}).
$
\end{theorem}

\subsection{Example in the SPD Matrices Cone}\label{Sec:ExampleSDP}

In this section, some examples of metrically regular mappings $\Phi$ on a particular manifold ${\cal M}$ are presented.

Denote the set of symmetric matrices of size $n \times n$ for some $n\in \mathbb{N}$ by $ {\cal P}(n)$, and   the cone of symmetric positive definite matrices by ${\cal M}= {\cal P}_{+}(n) $. 
 The latter is endowed with the affine invariant Riemannian metric given by   
\begin{equation}\label{eq:metric}
	\langle v,u \rangle\coloneqq \mbox{tr} (vp^{-1}up^{-1}),\qquad p\in \mathcal{M}, \quad v,u\in
	\mathcal{T}_p\mathcal{M},
\end{equation}
where $\mbox{tr}$ denotes the trace of a matrix. The tangent space $\mathcal{T}_p\mathcal{M}$ can be identified with $ {\cal P}(n)$.
Moreover,  $\mathcal{M}$ is a Hadamard manifold,  see,  for example, \cite[Theorem 1.2. p. 325]{Lang1999}.
The Riemannian distance is given by 
\begin{equation}\label{eq:KarcherMeanSDP}
	d( p,q) =  \tr^{ \frac{1}{2} } ( \ln^2 ( p^{-\frac{1}{2}}qp^{-\frac{1}{2}} ) ) , \qquad p,q \in \mathcal{M},
\end{equation}
where $\ln$ denotes the matrix logarithm. The Riemannian gradient of $f\colon \cal M \to \mathbb{R}$ at $p$ is given by
\begin{equation}\label{eq:riema.grad.spdm}
	\grad f(p)= p f'(p) p,
	\end{equation}
where $ f'(p) $ represents the Euclidean gradient of $f$ at $p$, see \cite{SraHosseini2015}. The identity matrix of size $n \times n$ will be denoted by $\id$. 

In the following two examples, we verify the property of metric regularity for single-valued functions.

\begin{example}\label{ex:1}
	Consider the function $\Phi\colon \mathcal{M} \to {\mathbb R}$ defined by $\Phi(p)= \ln(\tr(p))$. For every $x \in \mathbb{R}$, the following equality holds:
	\begin{equation}\label{eq.calapex1and3}
	\Phi ( e^{x - \ln({\tr(p)})} p )
	= \ln (  \tr( e^{x - \ln({\tr(p)})} p)) 
	=  \ln (  e^{x - \ln({\tr(p)})} \tr ( p ) ) 
	 =x,   \qquad   p  \in \mathcal{M}.
	\end{equation}
	This implies that   
	$\{e^{x- \ln({\tr(p)}) } p \colon p  \in \mathcal{M} \}  \subset \Phi^{-1}(x)$ for all $x\in \mathbb{R}.$ Thus, 
	using \eqref{eq:KarcherMeanSDP}  and simple algebraic manipulations, we obtain
	\begin{align*}
	d( \Phi^{-1}(x), q) &\leq d( e^{x- \ln({\tr(q)}) }  q , q),     \\
	&=  \tr^{\frac{1}{2}}(  | x - \ln({\tr(q)})  |^2 \id),  \\
	&= \sqrt{n} \, |  x - \ln({\tr(q)}) | =  \sqrt{n} \, |  x - \Phi(q) |
	, \qquad  x\in \mathbb{R}, \quad  q  \in \mathcal{M}.
	\end{align*}
	On the other hand, since $\Phi$ is Euclidean differentiable, it follows from \eqref{eq:riema.grad.spdm} that $\Phi$ will also be Riemannian differentiable with respect to metric \eqref{eq:metric}. Therefore, $\Phi$ is continuous and, by Proposition~\ref{prop:gph.closed}, it has closed graph. Given this and the last inequality above, we conclude that $\Phi$ is metrically regular at $p$ for $\Phi(p)$ for all $p\in {\cal M}$.
\end{example}

\begin{example} \label{ex:2}
	Define the function $\Phi\colon {\cal M} \to {\mathbb R}$ by $\Phi(p)=1/\tr(p)$. Note that
      $\{( x \tr(p) )^{-1} p \colon p  \in \mathcal{M} \}  \subset \Phi^{-1}(x)$ for all $x\in \mathbb{R} \backslash \{0\} $.
	Using this and  \eqref{eq:KarcherMeanSDP}, we obtain
	\begin{align}\label{eq:ex2.eqli}
	d( \Phi^{-1}(x), p) &\leq d( (x \tr(p) )^{-1} p , p)  \nonumber \\
	&= \tr^{\frac{1}{2}}(\ln^2( x \tr(p)\id))= \sqrt{n} \, |\ln(x \tr(p) )|
	,  \qquad  x\in \mathbb{R} \backslash \{0\}, \quad  p  \in \mathcal{M}.
	\end{align}  
      Now recall that 
	\begin{equation}\label{eq:ex3.eqli}
		\tr(p) = \lambda_1(p) + \ldots + \lambda_n(p),
	\end{equation}
	where $\lambda_1(p),\ldots,\lambda_n(p)$ are the eigenvalues of $p$. Pick a constant $a>0$. From \eqref{eq:KarcherMeanSDP}, for all  $i\in \{1,\ldots,n\}$, one has
     \begin{equation*}	
     |\ln(\lambda_i(p))| <	(\ln^2(\lambda_1(p)) + \ldots + \ln^2(\lambda_n(p)))^{\frac{1}{2}} =  \tr^{\frac{1}{2}}(\ln^2 (p))= d(p , \id ) < a, \qquad p \in \B_{a}(\id).
     \end{equation*}
 Then $e^{- a} <  \lambda_i(p)  < e^{a}$ holds for all $p \in \B_{a}(\id)$ and $i=1,\ldots,n$. Consequently
$$
ne^{- a} <  \lambda_1(p)  + \ldots + \lambda_n(p)< ne^{a}, \qquad  p \in \B_{a}(\id).
$$
With \eqref{eq:ex3.eqli}, we conclude that $1/\tr(p) \in (e^{-a}/n, e^{a}/n)$ for all $p \in \B_{a}(\id)$. Since the  function $\ln $ is Lipschitz on  the interval $(e^{-a}/n , e^{a}/ n)$ (its gradient is bounded on this interval), there exists $\sigma>0$ such that 
 \begin{equation}\label{ex2impdpe4}
 	|\ln(x \tr(p) )| =
	 \left| \ln(x)-\ln(  1/\tr(p)  ) \right|   \leq 
	\sigma \left| x-  1/\tr(p)   \right|    ,  \qquad  x\in (e^{-a}/n, e^{a}/n), \quad  p \in \B_{a}(\id).
\end{equation} 
Therefore, \eqref{eq:ex2.eqli} and \eqref{ex2impdpe4} yield
\begin{equation*}
	d( \Phi^{-1}(x), p)   \leq 
	\sqrt{n}\,\sigma \left| x-  \Phi(p)   \right|    ,  \qquad  x\in (e^{-a}/n, e^{a}/n), \quad  p \in \B_{a}(\id).
\end{equation*} 
Furthermore, taking into account the argument presented at the end of Example~\ref{ex:1}, it can be concluded that $\gph\Phi$ is closed. Then $\Phi$ is metrically regular  at $p$ for $\Phi(p)$ for all $p \in \B_{a}(\id)$. 
\end{example}

We now present two examples involving set-valued mappings. In these examples, we demonstrate that the given mappings are metrically regular at specific points in \({\cal M}\).

\begin{example}\label{ex:3}
	Define the set-valued mapping $\Phi\colon \mathcal{M} \rightrightarrows {\mathbb R}$ by 
	\begin{equation}\label{eqref:setvalued.ex}
	\Phi(p) =
	\begin{cases}
		\,\, \ln(\tr(p)), & p \in {\cal M} \backslash (1/n\id),
		\\
		 \{0\} \cup [1,2], & p= 1/n\id.
	\end{cases}
\end{equation}
By \eqref{eqref:setvalued.ex} and the calculations in  \eqref{eq.calapex1and3}, if $x\in \mathbb{R}$ and $q  \in \mathcal{M}$ satisfy $e^{x - \ln({\tr(q)})} q  \in \mathcal{M} \backslash (1/n\id)$ then $\Phi ( e^{x - \ln({\tr(q)})} q )
	\ni \ln (  \tr( e^{x - \ln({\tr(q)})} q)) 
	=x$. Now, consider the case $e^{x - \ln({\tr(q)})} q  = 1/n\id$. By applying $\tr$ on both sides, it follows that  $e^{x - \ln({\tr(q)})} \tr(q)  = 1$. Clearly, this implies that $x=0$. Therefore, by \eqref{eqref:setvalued.ex}, we get $\Phi ( e^{x - \ln({\tr(q)})} q ) \ni x$ for this case as well. Thus,
	$\{e^{x- \ln({\tr(q)}) } q \colon q  \in \mathcal{M} \}  \subset \Phi^{-1}(x)$ holds for all $x\in \mathbb{R}$, which implies
		\begin{align*}
	d( \Phi^{-1}(x), q) &\leq d( e^{x- \ln({\tr(q)}) }  q , q),     \\
	&=  \tr^{\frac{1}{2}}(  | x - \ln({\tr(q)})  |^2 \id) = \sqrt{n} \, |  x - \ln({\tr(q)}) | 
	, \qquad  x\in \mathbb{R}, \quad  q  \in \mathcal{M}.
	\end{align*}
	On the other hand, it follows from \eqref{eqref:setvalued.ex} that 
	$$
         |  x - \ln({\tr(q)}) | =  d_e(  x , \Phi(q) ), \qquad  x\in (-\infty,1/2), \quad  q  \in \mathcal{M}.
	 $$
	Overall, we conclude
	$$
	   d( \Phi^{-1}(x), q) \leq  \sqrt{n} \, d_e(  x , \Phi(q) ), \qquad  x\in (-\infty,1/2), \quad  q  \in \mathcal{M}.
	   $$
       With a justification analogous to the one presented in Example~\ref{ex:1}, we state that $\Phi$ has closed graph. Therefore,  $\Phi$ is metrically regular at $p$ for $\ln(\tr(p))$ for all $p\in \{ q\in {\cal M} \colon \ln({\tr(q)}) \in (-\infty,1/2)\}$.
\end{example}

	\begin{example} \label{ex:4}
         Consider the set-valued mapping $\Phi\colon \mathcal{M} \rightrightarrows {\mathbb R}$ by 
\begin{equation}\label{eqref:setvalued.ex4}
	\Phi(p) =
	\begin{cases}
		\,\,\,\,  1/\tr(p), & p \in {\cal M} \backslash \id,
		\\ 
		 \{1/n\} \cup [2,3], & p= \id.
	\end{cases}
\end{equation}
Following the steps of Example~\ref{ex:3}, we can show that  $
	\Phi ( 1/ ( x \tr(q) ) \, q )
	\ni 1/\tr(1/ ( x \tr(q) ) \, q) 
	=x$ holds for all $x\in \mathbb{R} \backslash \{0\} $, $q  \in \mathcal{M}$.
      This means that $\{ 1/ ( x \tr(q) ) \, q \colon q  \in \mathcal{M} \}  \subset \Phi^{-1}(x)$ for all $x\in \mathbb{R} \backslash \{0\} $, and hence
	\begin{align*}
	d( \Phi^{-1}(x), q) &\leq d( (x \tr(q) )^{-1} q , q)   \\
	&= \tr^{\frac{1}{2}}(\ln^2( x \tr(q)\id)) = \sqrt{n} \, |\ln(x \tr(q) )|
	,  \qquad  x\in \mathbb{R}\backslash \{0\}, \quad  q  \in \mathcal{M}.
	\end{align*}  
	From 
	 \eqref{ex2impdpe4},  for all  $a>0$ there exists $\sigma>0$ such that
        \begin{equation*}\label{eximpdpe4}
 	 d( \Phi^{-1}(x), q) \leq 
	\sigma \left| x-  1/\tr(q)   \right|    ,  \qquad  x\in (e^{-a}/n, e^{a}/n), \quad  q \in \B_{a}(\id).
\end{equation*} 
Pick $a>0$. Exploiting the definition of $\Phi$ in \eqref{eqref:setvalued.ex4}, it follows that
  \begin{equation*}\label{eximpdpe4mri}
 	 d( \Phi^{-1}(x), q) \leq 
	\sigma d_e( x,  \Phi(q)  )   ,  \qquad  x\in (e^{-a}/n, \min\{e^{a}/n,1 \}), \quad  q \in \B_{a}(\id).
\end{equation*} 
Combining this with the fact that 
 $\gph\Phi$ is closed, we conclude that $\Phi$ is metrically regular  at $p$ for $1/\tr(p)$ for all $p\in \{ q\in {\cal M} \colon 1/\tr(q) \in (e^{-a}/n, \min\{e^{a}/n,1 \}) \}$.
\end{example}

\section{Convergence}\label{sec:conv}

In this section, we will establish three convergence theorems for the sequence generated by the inexact Newton method described in \eqref{eq:INM}. The proofs presented here rely on a Riemannian version of Theorem 5G.3 from \cite{DontchevRockafellar2014}, which is presented in the following lemma. The proof of this lemma is a straightforward adaptation of its Euclidean version and will be omitted.
\begin{lemma}\label{pert.reg}
	Consider a set-valued mapping $F\colon \mathcal{M} \rightrightarrows \mathbb{R}^m$ and a point $(\bar p, \bar x)\in \gph F$ at which $F$ is metrically regular with positive constants $\sigma$, $a$ and $b$  satisfying \eqref{eq:loc.clo.me.re.def}-\eqref{def:MR}. Let $\nu>0$ be such that $\sigma \nu < 1$ and  $\sigma'>\sigma /(1-\sigma \nu)$. Then for every positive $\alpha$ and $\beta$ such that
	$$
	\alpha \leq a/2, \quad \nu \alpha + 2 \beta \leq b, \quad 2 \sigma' \beta \leq \alpha 
	$$
	and for every function $g \colon {\cal M} \to \mathbb{R}^m$ satisfying 
	\begin{equation*}
		\| g(\bar{p}) \|_{e} \leq \beta, \quad \|g(q) - g(q')\|_e \leq \nu d(q,q'),  \quad  q,q' \in {\cal B}_{\alpha}[\bar{p}],
	\end{equation*}
	the mapping $g+F$ has the following property: for every $x,x'\in \mathbb{B}_{\beta}[\bar x]$ and every $p\in(g+F)^{-1}(x)\cap {\cal B}_{\alpha}[\bar p]$ there exists $p'\in (g+F)^{-1}(x')$ such that
	$
	d(p,p')\leq  \sigma'  \| x - x' \|_e.
	$
\end{lemma}

Our first convergence result is a local analysis of the inexact Newton method \eqref{eq:INM} for solving \eqref{eq:mainproblem}. This approach makes assumptions around a solution $\bar{p}$ of \eqref{eq:mainproblem}.

\begin{theorem} \label{teo:main}
	Let  $\bar{p}$ be a solution of  \eqref{eq:mainproblem}. Suppose that the following conditions hold:
	\begin{enumerate}
		\item[(i)]  the function $f \colon {\cal M} \to \mathbb{R}^m$ is continuously differentiable at $\bar{p}$
		 and $\reg(f + F; \bar p| 0)= \sigma$;
		\item[(ii)] the sequence $\{R_k\}$ satisfies
		\begin{equation}\label{eq:errorcondition}
			\underset{\substack{p\to \bar{p}\\p\neq\bar{p}}}{\limsup }  \left\{  \frac{1}{d(p,\bar{p})}  \sup_{k\in \mathbb{N}_0} \sup_{x \in R_k(p)} \|x\|_e \right\} < \frac{1}{\sigma} , \qquad k=0, 1,\ldots.
		\end{equation}
	\end{enumerate}
	Then there exist $\theta \in (0,1)$ and a totally normal ball $\B_r(\bar{p}) \subset {\cal M}$
	 such that for every $p\in \B_r(\bar{p})\backslash \{\bar{p}\}$, every $k\in \mathbb{N}_0$, and every 
	$u_k\in R_k(p)$ there exists $q' \in \B_r(\bar{p})$ satisfying
	\begin{equation}\label{eq:maintheouk}
		f(p) + {\cal D}f(p)[\exp^{-1}_{p} q']+F(q')\ni u_k
	\end{equation}
	and
	\begin{equation}\label{eq:theomainconver}
		d(q',\bar{p})\leq \theta  d(p,\bar{p}).
	\end{equation}
	Consequently, for any starting point $p_0\in\B_{r}(\bar{p})$, there exists a sequence $\{p_k\}$ generated by \eqref{eq:INM} that converges linearly to $\bar{p}$.
\end{theorem}
\begin{proof}
	First, \eqref{eq:errorcondition} implies that there exists $\iota> 0$ satisfying the following inequality:
	\begin{equation}\label{proof:mainteo.1}
		\underset{\substack{p\to \bar{p}\\p\neq\bar{p}}}{\limsup }  \left\{  \frac{1}{d(p,\bar{p})}  \sup_{k\in \mathbb{N}_0} \sup_{x \in R_k(p)} \|x\|_e \right\} < \iota < \frac{1}{\sigma}.
	\end{equation}
	Choose $\iota$ satisfying \eqref{proof:mainteo.1} and positive constants  $\mu>0$, $\kappa > \sigma$, $\epsilon>0$ and $\theta \in (0,1)$ such that
	\begin{equation}\label{eq:parameters}
		\mu \kappa < 1 , \qquad \kappa (\epsilon + \iota ) <  \theta (1 - \mu\kappa).
	\end{equation}
	Pick $\Theta > 0$ and $\tau \in (\sigma,\kappa)$ such that
	\begin{equation}\label{eq.theta.bound}
		\frac{\sigma }{1 - \mu \sigma } < \Theta < \frac{\kappa}{1-\mu\kappa}, \qquad \frac{\tau }{1 - \mu \tau } < \Theta.
	\end{equation}
	From the first inequality in \eqref{proof:mainteo.1}, there exists $\delta>0$ such that
	\begin{equation}\label{eq:inexactineq}
		\|x\|_e < \iota d(p,\bar{p}) \quad \mbox{whenever} \quad
		p\in {\cal B}_{\delta}(\bar{p}) \backslash \{\bar{p}\}, \quad  x\in R_{k}(p), \quad k\in \mathbb{N}_0.
	\end{equation}
	Choose a totally normal ball $\B_{\bar{\delta}}(\bar{p}) \subset {\cal B}_{\delta}(\bar{p})$. Note that the function $g_p\colon {\cal M} \to \mathbb{R}^m$  defined by 
	\begin{equation}\label{def:littlegp}
	g_p(q) =
	\begin{cases}
	 {\cal D}f(p) [\exp_{p}^{-1}q - \exp_{p}^{-1}\bar{p}] -  {\cal D}f(\bar{p}) [ \exp_{\bar{p}}^{-1} q ]  , & q \in \B_{\bar{\delta}}(\bar{p}),
		\\
		\qquad 0, & q\in {\cal M} \backslash \B_{\bar{\delta}}(\bar{p}).
	\end{cases}
\end{equation}
	satisfies $g_p(\bar{p}) = 0$ for all $p \in \B_{\bar{\delta}}(\bar{p})$. On the other hand, since
	$f$ and $F$ satisfy $(i)$, it follows from Theorem~\ref{theo:lyus-graves-coro} that $\reg(G; \bar p| 0)=\sigma$, where
	$G\colon {\cal M} \rightrightarrows \mathbb{R}^m$ is the function defined in \eqref{eq:mappauxG} with $r=\bar{\delta}$. Hence, by the definition of $\reg$ in Definition~\ref{de:majcon}, there exist $a>0$ and $b>0$ such that \eqref{eq:loc.clo.me.re.def} and \eqref{def:MR} are satisfied for $\Phi = G$ and $\sigma=\tau$. Moreover, by  Lemma~\ref{lem:assumptionproof} and \eqref{def:littlegp},  there exists $\alpha < \min\{ \bar{ \delta}, a / 2, b/ \mu \} $ such that 
	\begin{equation*}
		\| 	g_{p} (q) - 	g_{p} (q') \|_e = \|{\cal D}f(p) [\exp_{p}^{-1}q - \exp_{p}^{-1}q'] -  {\cal D}f(\bar{p})[ \exp_{\bar{p}}^{-1}q -  \exp_{\bar{p}}^{-1}q']\|_e  
		\leq \mu d(q,q'),
	\end{equation*}
	for all $p, q,q'\in \B_{\alpha}(\bar{p})$. Therefore, for every $p \in \B_{\alpha}(\bar{p})$ one has
	\begin{equation*}\label{eq:ins.lemp.gb.cl}
		g_p(\bar{p}) = 0, \quad 
		\| 	g_{p} (q) - 	g_{p} (q') \|_e  
		\leq \mu d(q,q'), \quad q,q' \in \B_{\alpha}[\bar{p}].
	\end{equation*}
	Overall, there exists $\beta>0$ (independent of $p$) such that 
	\begin{equation*}
		\alpha \leq a/2, \qquad \mu \alpha + 2 \beta \leq b, \qquad 2 \Theta\beta \leq \alpha,
	\end{equation*}
	and 
	\begin{equation*}\label{eq:ins.lemp.gb.cl.1}
		\|g_p(\bar{p})\|_e \leq \beta, \quad 
		\| 	g_{p} (q) - 	g_{p} (q') \|_e  
		\leq \mu d(q,q'), \quad  p, q, q' \in \B_{\alpha}[\bar{p}].
	\end{equation*}
	Applying Lemma~\ref{pert.reg} with
	$g=g_p$, $F=G$,  $\nu=\mu$, $\sigma' = \Theta$ and $\bar{x}=0$, we conclude that 
	for every $x,x'\in \mathbb{B}_{\beta}[0]$ and every $ q \in(g_p + G)^{-1}(x)\cap \B_{\alpha}[\bar p]$ there exists $q'\in (g_p + G)^{-1}(x')$ such that
	$
	d(q,q')\leq  \Theta  \| x - x' \|_e.
	$
	Taking $x=\bar{x}=0$ and $q=\bar{p}$, we conclude that 
	for every $x' \in \mathbb{B}_{\beta}[0]$ there exists  $q' \in (g_{p} + G)^{-1}(x')$ such that
	\begin{equation}\label{eq.bytheo1}
		d( \bar{p} , q' ) \leq \Theta \|x'\|_e .
	\end{equation}

	Using the second part of Proposition~\ref{lem:properepsil} with  $\epsilon>0$ chosen as in \eqref{eq:parameters}, we have the existence of a normal ball $\B_{\delta_{\epsilon}}(\bar p)$ such that 
	\begin{equation}\label{eq:properepsi}
		\| f(p) - f(\bar{p}) - {\cal D}f(\bar{p})[\exp^{-1}_{\bar p} p]\|_e \leq \epsilon d(p,\bar{p}), \qquad  p \in \B_{\delta_{\epsilon}}(\bar p),
	\end{equation}
	Fixed $r$ such that 
	\begin{equation}\label{eq:defdraioconv}
	0< r < \min\left\{ \frac{\beta}{\epsilon + \iota},\bar{\delta}, \delta_{\epsilon}\right\}
	\end{equation}
	and  $p\in \B_{r}(\bar{p}) \backslash \{\bar{p}\}$, by choosing $k\in \mathbb{N}_0$ and $u_k\in R_k(p)$, it follows from \eqref{eq:inexactineq} that $u_k$ satisfies $	\|u_k\|_e \leq \iota d(p,\bar{p}) $. Denote
	\begin{equation}\label{eq:yk}
		y_k \coloneqq f(p) - f(\bar{p}) + {\cal D}f(p)[\exp^{-1}_{p} \bar{p}] - u_k.
	\end{equation}
	If $y_k=0$, then $q' \coloneqq \bar{p}$ satisfies \eqref{eq:maintheouk} because $-f(\bar{p})\in F(\bar{p})$, and \eqref{eq:theomainconver} holds trivially. For $y_k\neq 0$, it follows from \eqref{eq:yk},  \eqref{eq:properepsi}  and \eqref{eq:inexactineq} that
	\begin{equation}\label{eq:ykbounded}
		\|y_k\|_e \leq \| f(p) - f(\bar{p}) + {\cal D}f(p)[\exp^{-1}_{p} \bar{p}] \|_e + \| u_k\|_e \leq (\epsilon + \iota)d( p ,\bar{p}).
	\end{equation}
	Since $ d(p,\bar{p}) < r < \beta/(\epsilon + \iota)$, it follows from \eqref{eq:ykbounded} that $\|y_k\|_e < \beta$. Applying \eqref{eq.bytheo1} with $x'=-y_k$, we obtain that there exists $q'\in (g_{p} + G)^{-1}(-y_k)$ such that $d(q',\bar{p}) \leq \Theta \|y_k\|_e $. Hence, from the upper bound for $\Theta$ given in \eqref{eq.theta.bound}, \eqref{eq:ykbounded}, and the last inequality in \eqref{eq:parameters}, it follows that 
	\begin{equation*}
		d(q',\bar{p}) < \frac{\kappa}{1-\mu\kappa} \|y_k\|_e <  \frac{\kappa(\epsilon + \iota)   }{1-\mu\kappa} d(p,\bar{p})< \theta d(p,\bar{p}).
	\end{equation*}	
	Furthermore, it comes from $-y_k \in (g_p + G)(q')$,  \eqref{eq:yk}, \eqref{def:littlegp} and \eqref{eq:mappauxG} that
	$$
	-f(p) + f(\bar{p}) - {\cal D}f(p)[\exp^{-1}_{p} \bar{p}] + u_k  \in  f(\bar{p}) + 
	{\cal D}f(p) [ \exp_{p}^{-1} q' - \exp_{p}^{-1}\bar{p} ] + F(q'),
	$$
	which means that $
	u_k  \in   f(p) +
	{\cal D}f(p) [ \exp_{p}^{-1} q']  + F(q').$ Thus, $q'$ satisfies \eqref {eq:maintheouk} and \eqref{eq:theomainconver}.

	Now choose $p_0\in \B_r(\bar{p})$. If $p_k=\bar{p}$, then $p_{k+1}\coloneqq \bar{p}$ verifies \eqref{eq:INM} because $0\in R_{k}(\bar{x})$.  If $p_k\neq \bar{p}$, applying \eqref {eq:maintheouk} and \eqref{eq:theomainconver} with $p=p_0$ and $q' = p_{1}$ we have
	$
	f(p_0)+ {\cal D}f(p_0)[\exp^{-1}_{p_0} p_{1}]  + F(p_{1})\ni u_0
	$
	and 
	\begin{equation*}\label{eq:theomainiterconv}
		d(p_{1},\bar{p})\leq \theta  d(p_0,\bar{p}).
	\end{equation*}
	Repeating this argument one can conclude that there exists a sequence $\{p_k\}$ in  $\B_{r}(\bar{p})$  which satisfies \eqref{eq:INM} and converges linearly to $\bar{p}$.
\end{proof}


An upper bound for the radius $r$ of the ball mentioned in the statement of Theorem~\ref{teo:main} is given in \eqref{eq:defdraioconv}. This bound depends on certain constants that, while not known a priori, are guaranteed to exist due to conditions i) and ii) in Lemma~\ref{pert.reg}, as well as additional results presented in the Appendix of this paper.

Below, we present a version of Theorem~\ref{teo:main} with more technical conditions, where it is assumed that the constants mentioned above are known. Under these new conditions, we can explicitly determine a radius of convergence for the sequence ${p_k}$ generated by \eqref{eq:INM}. The proof of this new theorem follows a similar structure to the proof of Theorem~\ref{teo:main} and will therefore be omitted.
 

\begin{theorem}\label{teo:main.ovc}
Let $\bar{p}$ be a solution to the equation \eqref{eq:mainproblem}. We assume that the following conditions are satisfied:
\begin{enumerate}
\item[(i)] $\B_{\bar{\delta}}(\bar{p})$ is a totally normal ball, and the mapping $G\colon {\cal M} \rightrightarrows \mathbb{R}^m$ defined in \eqref{eq:mappauxG} with $r=\bar{\delta}$ is metrically regular at $\bar{p}$ for $0$ with the constant $\tau>0$ and neighborhoods ${\cal B}_a[\bar{ p}]$ and $\mathbb{B}_b[0]$.
	\item[(ii)]  $\iota\in (0,1/\tau)$  is chosen such that
		$\|x\|_e < \iota d(p,\bar{p})$ holds for all $p\in \B_{\bar{\delta}}(\bar{p}) \backslash \{\bar{p}\}$,  $x\in R_{k}(p)$,  $k\in \mathbb{N}_0$.
	\item[(iii)] $\mu>0$, $\kappa > \tau$, $\epsilon>0$ and  $\Theta > 0$  satisfy
	\begin{equation*}\label{eq:parameters.rem}
		\mu \kappa < 1 , \qquad \kappa (\epsilon + \iota ) < 1 - \mu\kappa, \qquad \frac{\tau }{1 - \mu \tau } < \Theta < \frac{\kappa}{1-\mu\kappa}.
	\end{equation*}
	\item[(iv)] $\alpha < \min\{ \bar{ \delta}, a / 2, b/ \mu \} $ and 
	\begin{equation*}\label{eq:desmant3}
		 \|{\cal D}f(p) [\exp_{p}^{-1}q - \exp_{p}^{-1}q'] -  {\cal D}f(\bar{p})[ \exp_{\bar{p}}^{-1}q -  \exp_{\bar{p}}^{-1}q']\|_e  
		\leq \mu d(q,q') \quad    \mbox{for all }\,     p, q,q'\in \B_{\alpha}(\bar{p}).
	\end{equation*}
	\item[(v)] $\beta>0$ is such that $\mu \alpha + 2 \beta \leq b$ and $2 \Theta\beta \leq \alpha$.
         \item[(vi)] $\delta_{\epsilon}>0$ satisfies 
		$\| f(p) - f(\bar{p}) - {\cal D}f(\bar{p})[\exp^{-1}_{\bar p} p]\|_e \leq \epsilon d(p,\bar{p})$ for all  $p \in \B_{\delta_{\epsilon}}(\bar p)$.
\end{enumerate}
Then, for every starting point $p_0\in \B_{r}(\bar{p})$ with  $0< r < \min\{\beta/(\epsilon + \iota),\bar{\delta}, \delta_{\epsilon}\}$, there exists a sequence $\{p_k\}$ generated by \eqref{eq:INM} that is well defined and converges linearly to $\bar{p}$.
\end{theorem}

Next, we modify the assumption on the multifunction $R_k$ and introduce a stronger condition for the differentiability of the function $f$ in order to establish quadratic convergence of the sequence \eqref{eq:INM}.
\begin{theorem}\label{teo:main2}
	Let  $\bar{p}$ be a solution of  \eqref{eq:mainproblem}. Suppose the following conditions hold:
	\begin{enumerate}
		\item[(i)] the function $f \colon {\cal M} \to \mathbb{R}^m$ is continuously differentiable at $\bar{p}$, and $\reg(f + F; \bar p| 0)= \sigma$;
		\item[(ii)] ${\cal D} f$ is
		L-Lipschitz continuous around $\bar{p}$;
		\item[(iii)]  the sequence $\{R_k\}$ satisfies
		\begin{equation}\label{eq:errorcondition2}
			\underset{\substack{p\to \bar{p}\\p\neq\bar{p}}}{\limsup }  \left\{  \frac{1}{d^2(p,\bar{p})}  \sup_{k\in \mathbb{N}_0} \sup_{x \in R_k(p)} \|x\|_e \right\} < \frac{1}{\sigma}, \qquad k=0, 1,\ldots.
		\end{equation}
	\end{enumerate}
	Then there exist $\theta>0$ and a totally normal ball $\B_r(\bar{p}) \subset {\cal M}$
 such that for every $p\in \B_r(\bar{p})\backslash \{\bar{p}\}$, every $k\in \mathbb{N}_0$, and every 
	$u_k\in R_k(p)$, there exists $q'\in \B_r(\bar{p})$ satisfying
	\begin{equation}\label{eq:maintheouk2}
		f(p)+{\cal D}f(p)[\exp^{-1}_{p} q'] +F(q')\ni u_k
	\end{equation}
	and
	\begin{equation}\label{eq:theomainconver2}
		d(q',\bar{p})\leq \theta  d^2(p,\bar{p}).
	\end{equation}
	Consequently, for any starting point $p_0\in \B_{r}(\bar{p})$, there exists a sequence $\{p_k\}$ generated by \eqref{eq:INM} that converges quadratically to $\bar{p}$.
\end{theorem}
\begin{proof}
	This proof is analogous to the proof of Theorem~\ref{teo:main}. Using \eqref{eq:errorcondition2}, we can find $\iota<1/\sigma $
	and  a totally normal ball $\B_{\delta}(\bar p)$ such that
	\begin{equation}\label{eq:inexactineq2}
		\|x\|_e < \iota d^2(p,\bar{p}) \quad \mbox{whenever} \quad
		p\in \B_{\delta}(\bar{p}) \backslash \{\bar{p}\},\quad k\in \mathbb{N}_0, \quad  x\in R_{k}(p).
	\end{equation}
         Choose  $\mu>0$, $\kappa > \sigma$, $\Theta >0$ and $\tau \in (\sigma,\kappa)$ satisfying
	\begin{equation}\label{eq:parameters2}
		\mu \kappa < 1, \qquad \frac{\sigma }{1 - \mu \sigma } < \Theta < \frac{\kappa}{1-\mu\kappa}, \qquad \frac{\tau }{1 - \mu \tau } < \Theta.
	\end{equation}
	Note that $\mu \tau < 1$. From $(i)$ and Theorem~\ref{theo:lyus-graves-coro}, we obtain $\reg(G; \bar p| 0)=\sigma$, where
	$G\colon \mathcal{M} \rightrightarrows \mathbb{R}^m$ is the function defined in \eqref{eq:mappauxG} with $r=\delta$. Thus, there exist $a > 0$ and $b > 0$ such that $G$ is metrically regular at $\bar{p}$ for $0$
with the constant $\tau$ and neighborhoods ${\cal B}_a[\bar{ p}]$ and $\mathbb{B}_b[0]$. Now, consider the auxiliary functions $g_p\colon {\cal M} \to \mathbb{R}^m$, $p \in \B_{\delta}(\bar{p})$, defined in \eqref{def:littlegp}. Recall that $g_p(\bar{p})=0$ for all $p \in \B_{\delta}(\bar{p})$.
         Make $\delta > 0$ smaller, if necessary, to ensure
	\begin{equation*}
		\| 	g_{p} (q) - 	g_{p} (q') \|_e 
		\leq \mu d(q,q'), \qquad p, q,q'\in \B_{\delta}[\bar{p}].
	\end{equation*}
	Hence, we can apply Lemma~\ref{pert.reg} with $g=g_p$, $F=G$, $\sigma = \tau$, $\nu=\mu$, $\sigma'=\Theta$, $x=\bar{x}=0$, and $p=\bar{p}$. This yields the existence of $\beta > 0$ (independent of the point $p\in \mathcal{B}_{\delta}(\bar{p})$ that determines the function $g_p$) such that for each $x' \in \mathbb{B}_{\beta}[0]$, there exists $q' \in (g_{p} + G)^{-1}(x')$ such that
	\begin{equation}\label{eq.bytheo12}
		d( q' ,\bar{p}) \leq \Theta \|x'\|_e.
	\end{equation}
	With $(ii)$, the last part of Proposition~\ref{lem:properepsil} implies that there exists $\delta_L > 0$ such that
	\begin{equation}\label{eq:lemcin-llips.tcq}
		\| f(p) - f(\bar{p}) - {\cal D}f(\bar{p}) [\exp^{-1}_{\bar p} p] \|_e \leq L d^2(p,\bar{p}), 
		\qquad  p \in \B_{\delta_L}(\bar p).
	\end{equation}
	Fixed $r$ such that 
	\begin{equation}\label{eq:rdcqc}
	0< r < \min\left\{\left(\frac{\beta}{L + \iota}\right)^{1/2}, \frac{1-\mu \kappa}{\kappa(L+\iota)}, \;\delta, \;\delta_L\right\}
	\end{equation}
	and $p\in \B_{r}(\bar{p}) \backslash \{\bar{p}\}$ fixed, by choosing $k\in \mathbb{N}_0$ and $u_k\in R_k(p)$ it follows from \eqref{eq:inexactineq2} that $u_k$ satisfies $	\|u_k\|_e \leq \iota d^2(p,\bar{p}) $. Denote
	\begin{equation}\label{eq:yk2}
		y_k \coloneqq f(p) - f(\bar{p}) + {\cal D}f(p)[\exp^{-1}_{p} \bar{p}] - u_k.
	\end{equation}
	If $y_k=0$, then $q' \coloneqq \bar{p}$ satisfies \eqref{eq:maintheouk2} because $-f(\bar{p})\in F(\bar{p})$ and \eqref{eq:theomainconver2} hold trivially. Assume that $y_k\neq 0$. Using \eqref{eq:lemcin-llips.tcq} and \eqref{eq:yk2}, we obtain
	\begin{equation}\label{eq:ykbounded2}
		\|y_k\|_e \leq \| f(p) - f(\bar{p}) + {\cal D}f(p)[\exp^{-1}_{p} \bar{p}] \|_e + \| u_k\|_e \leq (L + \iota)d^2( p ,\bar{p}).
	\end{equation}
	Since $ d(p,\bar{p}) < r < \left(\beta/(L + \iota)\right)^{1/2}$, it follows from \eqref{eq:ykbounded2} that $\|y_k\|_e < \beta$. Applying \eqref{eq.bytheo12} with $x'=-y_k$, we obtain that there exists $q'\in (g_{p} + G)^{-1}(-y_k)$ such that $d(q',\bar{p}) \leq \Theta \|y_k\|_e $. Hence, utilizing the upper bound for $\Theta$ given in \eqref{eq:parameters2} and \eqref{eq:ykbounded2}, it follows that 
	\begin{equation*}
		d(q',\bar{p}) \leq \frac{\kappa}{1-\mu\kappa} \|y_k\|_e \leq  \theta d^2(p,\bar{p}) \quad \mbox{where} \quad \theta:=\frac{\kappa(L + \iota)   }{1-\mu\kappa}.
	\end{equation*}
	Furthermore, it comes from $-y_k \in (g_p + G)(q')$,  \eqref{eq:yk2}, \eqref{def:littlegp} and \eqref{eq:mappauxG} that
	$$
	-f(p) + f(\bar{p}) - {\cal D}f(p)[\exp^{-1}_{p} \bar{p}] + u_k  \in  f(\bar{p}) + 
	{\cal D}f(p) [  \exp_{p}^{-1} q' - \exp_{p}^{-1}\bar{p} ] + F(q'),
	$$
which means that $
	u_k  \in   f(p) +
	{\cal D}f(p) [ \exp_{p}^{-1} q']  + F(q').$ Thus, $q'$ satisfies \eqref {eq:maintheouk2} and \eqref{eq:theomainconver2}.
	
	To finish the proof, choose any  $p_0\in \B_{r}(\bar{p})$. If $p_0 = \bar{p}$, then $p_{1} \coloneqq \bar{p}$ verifies \eqref{eq:INM} because $0\in R_0(\bar{p})$. If $p_0\neq \bar{p}$, applying \eqref {eq:maintheouk2} and \eqref{eq:theomainconver2} we obtain that for every $u_0\in R_0(p_0)$ there exists $p_{1}$  such that
	\begin{equation*}\label{eq:maintheoiter2}
		f(p_0)+{\cal D}f(p_0)[\exp^{-1}_{p_0} p_{1}] + F(p_{1})\ni u_0
	\quad \mbox{and} \quad
		d(p_{1},\bar{p})\leq \theta  d^2(p_0,\bar{p}).
	\end{equation*}
	By considering the definition of $r$ and $\theta,$ we obtain from the above inequality that $p_1\in \B_{r}(\bar{p}).$  Repeating the previous argument it is possible to construct a sequence $\{p_k\}$ in $\B_{r}(\bar{p})$ that satisfies \eqref{eq:INM} and converges quadratically to $\bar{p}$.
\end{proof}

Under suitable conditions, it is possible to determine the radius $r$ mentioned in Theorem~\ref{teo:main2}. Details are given in the following result. The proof of this result is along the same lines as the proof of Theorem~\ref{teo:main2} and will therefore be omitted.

\begin{theorem}\label{teo:main.ovcq}
Let  $\bar{p}$ be a solution of  \eqref{eq:mainproblem}. Suppose that the following conditions hold:
\begin{enumerate}
\item[(i)]  $\B_{\delta}(\bar{p})$ is a totally normal ball, and the mapping $G\colon {\cal M} \rightrightarrows \mathbb{R}^m$ defined in \eqref{eq:mappauxG} with $r=\delta$ is metrically regular at $\bar{p}$ for $0$ with the constant $\tau>0$ and neighborhoods ${\cal B}_a[\bar{ p}]$ and $\mathbb{B}_b[0]$.
	\item[(ii)]  $\iota\in (0,1/\tau)$  satisfies
		$\|x\|_e < \iota d^2(p,\bar{p})$ for all $p\in \B_{\delta}(\bar{p}) \backslash \{\bar{p}\}$,  $x\in R_{k}(p)$,  $k\in \mathbb{N}_0$.
	\item[(iii)] $\mu>0$, $\kappa > \tau$ and  $\Theta > 0$  satisfy
	\begin{equation*}\label{eq:parameters.rem}
		\mu \kappa < 1 ,  \qquad \frac{\tau }{1 - \mu \tau } < \Theta < \frac{\kappa}{1-\mu\kappa}.
	\end{equation*}
	\item[(iv)] $\alpha < \min\{  \delta, a / 2, b/ \mu \} $ and 
	\begin{equation*}\label{eq:desmant3}
		 \|{\cal D}f(p) [\exp_{p}^{-1}q - \exp_{p}^{-1}q'] -  {\cal D}f(\bar{p})[ \exp_{\bar{p}}^{-1}q -  \exp_{\bar{p}}^{-1}q']\|_e  
		\leq \mu d(q,q') \quad    \mbox{for all }\,     p, q,q'\in \B_{\alpha}(\bar{p}).
	\end{equation*}
	\item[(v)] $\beta>0$ satisfies $\mu \alpha + 2 \beta \leq b$ and $2 \Theta\beta \leq \alpha$.
         \item[(vi)] $L>0$ and $\delta_L>0$ satisfy
		$\| f(p) - f(\bar{p}) - {\cal D}f(\bar{p})[\exp^{-1}_{\bar p} p]\|_e \leq L d^2(p,\bar{p})$ for all  $p \in \B_{\delta_{L}}(\bar p)$.
\end{enumerate}
Then for every  $r$ satisfying \eqref{eq:rdcqc} and starting point $p_0\in \B_{r}(\bar{p})$ there exists a sequence $\{p_k\}$ generated by \eqref{eq:INM} that is well defined and converges linearly to $\bar{p}$.
\end{theorem}

Some comments about the previous results are in order. First, Theorems \ref{teo:main} and \ref{teo:main2} establish conditions for ensuring that the inexact Newton method \eqref{eq:INM} converges with linear and quadratic rates, respectively. Second, a similar result to Theorem \ref{teo:main} is presented in \cite[Theorem 2.1]{CibulkaDontchev2015} by considering ${\cal M}$ as Banach spaces. Thus, if ${\cal M}=X$, where $X$ is a Banach space, and the derivative ${\cal D}f(\cdot)$ is replaced by a suitable approximation, then Theorems \ref{teo:main} and \cite[Theorem 2.1]{CibulkaDontchev2015} are equivalent. Third, to ensure superlinear and quadratic convergence, the authors in \cite{CibulkaDontchev2015} assume a stronger condition, namely, the strong metric regularity (see \cite[Theorem 2.3]{CibulkaDontchev2015}). We recall that a set-valued mapping $F$ is strongly metrically regular at $\bar x$ for $\bar y$ if and only if its inverse $F^{-1}$ has a single-valued graphical localization around $\bar{y}$ for $\bar{x}$ which is Lipschitz continuous around $\bar{y}$ with Lipschitz modulus at $\bar{y}$ equal to $\operatorname{reg}(F; \bar{x}|\bar{y})$, see \cite[p. 1007]{CibulkaDontchev2015}.
Fourth, we establish quadratic convergence of \eqref{eq:INM} by only assuming the metric regularity condition, which is clearly a weaker assumption than strong metric regularity, see the previous comment. Thus, the result obtained in Theorem \ref{teo:main2} is stronger than \cite[Theorem 2.3]{CibulkaDontchev2015}. Finally, Theorems \ref{teo:main.ovc} and \ref{teo:main.ovcq} refine Theorems \ref{teo:main} and \ref{teo:main2}, respectively, in the sense that they provide guidance on how to find the neighborhood to start the proposed method in \eqref{eq:INM}.

It is also important to mention that in \cite{CibulkaDontchev2015} is introduced a version of the Dennis-Mor\'e theorem for the sequence generated by \eqref{eq:INM}. We do not go futher in this topic because we do not proposed a quasi-Newton method in this paper.

We conclude this section by presenting a semi-local convergence result for the inexact Newton method proposed in this paper to solve \eqref{eq:mainproblem}. This result makes no assumptions about the unknown solution to the problem under investigation; instead, the assumptions are made about the starting point $p_0$. It is worth mentioning that this result is novel, even in the Euclidean context.

\begin{theorem} \label{teo:semilocal}
Assume that for $(p_0,y_0)\in \Omega \times\mathbb{R}^m$ with $y_0\in f(p_0)+F(p_0)$ and $u_0\in R_0(p_0)$, the following conditions hold:
	\begin{enumerate}
		\item [(i)] $\B_{\delta}(p_0) \subset \Omega$ is a totally normal ball, and the mapping $G\colon {\cal M} \rightrightarrows \mathbb{R}^m$ defined by
		\begin{equation}\label{eq:mappauxG-sl}
	G(p) =
	\begin{cases}
	 f(p_0)+{\cal D}f(p_0)[\exp^{-1}_{p_0} p] + F(p), & p \in \B_{\delta}(p_0)
		\\
		f(p) + F(p), & p\in {\cal M} \backslash \B_{\delta}(p_0)
	\end{cases}
\end{equation} 
		 is metrically regular at $p_0$ for $y_0$ with the constant $\sigma>0$ and neighborhoods ${\cal B}_a[p_0]$ and $\mathbb{B}_b[y_0]$. 
		 \item[(ii)] the positive constants $\mu$, $\alpha$, $\beta$, $\Theta$, $\epsilon$ and $\iota$ satisfy 
		  $$
		\mu \sigma < 1, \qquad  \alpha \leq a/2, \qquad   \mu \alpha + 2 \beta \leq b,
		 $$
		 $$
		   \sigma /(1 - \mu \sigma ) < \Theta \leq \alpha/(2\beta),  \qquad  \epsilon + \iota  <   2\beta/\alpha.
		 $$
		  \item[(iii)] 
		   $\|u_0\|_{e}\le \iota\|y_0\|_e$ and    for all  
		  $p\in \B_{\delta}(p_0)$ there holds
		\begin{equation}\label{eq:y0.bound.s-l}
		 \|y_0\|_e \leq \min\left\{ \frac{\beta}{\Theta(1+\iota)}, \frac{\beta}{1+\iota},   \frac{\beta( 1 - \hat{\alpha} )}{\hat{\alpha} - \hat{\alpha}^2 + 1 },   \frac{\delta( 1 - \hat{\alpha} )}{\Theta(1+\iota)}  \right\}, \qquad \hat{\alpha} \coloneqq  \Theta ( \epsilon + \iota).
		 \end{equation}
		\item[(iv)]  
		$\|x\|_e + \|y\|_e \leq\iota d(\tilde p,\tilde q)$ holds for all $\tilde p,\tilde q\in \B_{\delta}(p_0)$,  $x\in R_{k}(\tilde p)$, $y \in R_{k-1}(\tilde q)$, $k\in \mathbb{N}$ and $\tilde p\neq \tilde q$.	
	      \item[(v)]$f \colon {\cal M} \to \mathbb{R}^m$ is continuously differentiable,  and the inequalities
		$$
		\| f(q) - f(p) - {\cal D}f( p)[\exp^{-1}_{p}  q]\|_e \leq \epsilon d( p, q)
		$$
		and
	          \begin{equation*}\label{eq:desmant3}
		 \|{\cal D}f( p) [\exp_{ p}^{-1} q - \exp_{ p}^{-1} q'] -  {\cal D}f(p_0)[ \exp_{p_0 }^{-1} q -  \exp_{p_0}^{-1} q']\|_e  
		\leq \mu d(q, q')
	         \end{equation*}	
	         hold for all $p, q, q'\in \B_{\delta}(p_0) $.
	\end{enumerate}
	Then there exist   sequences $\{p_{k}\}_{k\in \mathbb{N}}$ and $\{u_{k}\}_{k\in \mathbb{N} }$ satisfying:
	\begin{itemize}
		\item[(A1)] $d(p_k, p_0)\leq \frac{1-\hat{\alpha}^k}{1-\hat{\alpha}}\Theta(1+\iota) \|y_0 \|_e.$
		\item[(A2)] $d(p_k, p_{k-1})\leq \hat{\alpha}^{k-1} \Theta(1+\iota) \|y_0 \|_e.$
		\item[(A3)] $u_{k-1}\in f(p_{k-1}) + {\cal D}f(p_{k-1})[\exp^{-1}_{p_{k-1}} p_{k}]+F(p_{k}).$ 
	\end{itemize}
	In particular, $\{p_{k}\}_{ k\in \mathbb{N} }$ remains in $ \B_{\delta}(p_0)$, converges to a solution $\bar p $ of \eqref{eq:mainproblem}, and satisfies
	\begin{equation*}
	d(p_k, \bar{p}) \leq \frac{ \hat{\alpha}^k }{1-\hat{\alpha}}	\Theta(1+\iota)\|y_0\|_e \,\, \mbox{ for all } \,\, k\in \mathbb{N}.
\end{equation*}

\end{theorem}
\begin{proof}
Note that the function
         $g_p\colon {\cal M} \to \mathbb{R}^m$, $p \in \B_{\delta}(p_0)$,  defined by 
	\begin{equation}\label{def:littlegp-sl}
	g_p(q) =
	\begin{cases}
	 {\cal D}f(p) [\exp_{p}^{-1}q - \exp_{p}^{-1} p_0] -  {\cal D}f(p_0) [ \exp_{ p_0 }^{-1} q ],   & q \in \B_{ \delta }( p_0 )
		\\
		\qquad 0, & q\in {\cal M} \backslash \B_{\delta}(p_0),
	\end{cases}
\end{equation}
	satisfies $g_p(p_0) = 0$. From  \eqref{def:littlegp-sl} and the second inequality in $(v)$,  it follows that
	\begin{equation*}
		\| 	g_{p} (q) - 	g_{p} (q') \|_e = \|{\cal D}f(p) [\exp_{p}^{-1}q - \exp_{p}^{-1}q'] -  {\cal D}f(p_0)[ \exp_{p_0}^{-1}q -  \exp_{p_0}^{-1}q']\|_e  
		\leq \mu d(q,q'),
	\end{equation*}
	for all $p, q,q'\in \B_{\delta}(p_0)$. Thus, by using $(ii)$, one has
	\begin{equation*}
		\alpha \leq a/2, \qquad \mu \alpha + 2 \beta \leq b, \qquad 2 \Theta\beta \leq \alpha,
	\end{equation*}
	and 
	\begin{equation*}\label{eq:ins.lemp.gb.cl.1}
		\|g_p(p_0)\|_e < \beta, \quad 
		\| 	g_{p} (q) - 	g_{p} (q') \|_e  
		\leq \mu d(q,q'), \quad  p, q, q' \in {\cal B}_{\alpha}(p_0).
	\end{equation*}
	Hence, we can apply Lemma~\ref{pert.reg} with
	$g=g_p$, $F=G$,  $\nu=\mu$, $\sigma' = \Theta$, $\bar{x}=y_0$ and $\bar{p}=p_0,$ to obtain the following statement:
	\begin{itemize}
	\item[ \textbf{(S)} ]
	for every $x,x' \in \mathbb{B}_{\beta}[y_0]$ and every  $\tilde q \in(g_p + G)^{-1}(x)\cap \B_{\alpha}[p_0]$ there exists  $\tilde q'\in (g_p + G)^{-1}(x')$ such that $d(\tilde q,\tilde q')\leq  \Theta  \| x - x' \|_e$.
	\end{itemize}
	Due to $(iv)$, for every $k\in \mathbb{N}$, the following can be stated:
	\begin{equation}\label{eq:ine.ip.-sl}
	 p_{k}\neq p_{k-1} \in \B_{\delta}(p_0), \,u_{k} \in R_k(p_k)
	  \implies
	 \| u_k \|_e + \|u_{k-1}\|_e \leq \iota d(p_k,p_{k-1}).
	 \end{equation}
	We are now able to prove (A1)--(A3) by induction. 
	By using statement \textbf{(S)} with $x=y_0,$ $\tilde q=p=p_0,$ and $x'=u_0,$ we conclude that 
	there exists  $p_1 :=\tilde q'\in (g_{p_0} + G)^{-1}(u_0)$ such that
		\begin{equation*}\label{eq.bytheo1-sl}
		d(p_1 , p_0) \leq \Theta \|y_0 - u_0\|_e\le \Theta(1+\iota) \|y_0 \|_e,
	\end{equation*} 
where the last inequality follows from $(iii)$. 
Moreover, the inclusion $p_1 \in (g_{p_0} + G)^{-1}(u_0)$ is equivalent to
	$$
	u_0\in f(p_0)+{\cal D}f(p_0)[\exp^{-1}_{p_0} p_1]+F(p_1).
	$$
	Hence, (A1)--(A3) hold for $k=1$. 
	For $k>1$, assume the induction hypothesis:  there exist $p_j\in \B_{\delta}(p_0)$ and $u_j\in R_j(p_j)$ such that (A1)--(A3) hold for every $j\in \{1,2,\ldots,k\}$. Denote
	\begin{align}
		z_k & \coloneqq f(p_0)-f(p_k) - {\cal D}f(p_k)[\exp^{-1}_{p_k} p_0] + u_k,	\label{defzk-slc}\\
	        w_k & \coloneqq f(p_0)-f(p_{k-1})-{\cal D}f(p_k)[\exp^{-1}_{p_k} p_0]-{\cal D}f(p_{k-1})[\exp^{-1}_{p_{k-1}} p_k] + u_{k-1}. \label{defwk-slc}
	\end{align}
	By (A3), \eqref{def:littlegp-sl}, \eqref{eq:mappauxG-sl} we find $p_k\in (g_{p_{k}} + G)^{-1}(w_k)$. On the other hand, the first inequality in $(v)$, \eqref{eq:ine.ip.-sl}, (A1), (A2) and \eqref{eq:y0.bound.s-l} yield
	\begin{align*}
		\|z_k - y_0\|_e &\leq \| f(p_0) - f(p_k)  -{\cal D}f(p_k)[\exp^{-1}_{p_k} p_0] \|_e + \|u_k\|_e + \|y_0\|_e\\
		&\leq \epsilon  d(p_k,p_0) + \iota d(p_k,p_{k-1}) + \|y_0\|_e \\
		&\leq \epsilon  d(p_k,p_0) + (\epsilon + \iota) d(p_k,p_{k-1}) + \|y_0\|_e 
	\end{align*}
        and
	\begin{align*}
		\|w_k-y_0\|_e &\leq \|f(p_0)-f(p_{k}) - {\cal D}f(p_{k})[\exp^{-1}_{p_{k}} p_0] \|_e \\
		&+\|f(p_k)-f(p_{k-1})-{\cal D}f(p_{k-1})[\exp^{-1}_{p_{k-1}} p_k]\|_e + \|u_{k-1}\|_e + \|y_0\|_e\\
		&\leq \epsilon d(p_{k},p_0) +( \epsilon  + \iota ) d(p_{k},p_{k-1}) + \|y_0\|_e
	\end{align*}
	and
	\begin{align*}
		 \epsilon  d(p_k,p_0) + (\epsilon + \iota) d(p_k,p_{k-1}) + \|y_0\|_e 
		&\leq  \frac{\epsilon\Theta}{ 1 - \hat{\alpha} }\|y_0\| +( \epsilon  + \iota ) \Theta\|y_0\| + \|y_0\|_e\\
	          &\leq  \frac{(\epsilon + \iota)\Theta}{ 1 - \hat{\alpha} } \|y_0\| +( \epsilon  + \iota ) \Theta\|y_0\| + \|y_0\|_e\\
	          &=   \frac{ \hat{\alpha} - \hat{\alpha}^2+1 }{ 1 - \hat{\alpha} }  \|y_0\|_e \leq \beta,
	\end{align*}
	which implies $w_k, z_k\in\mathbb{B}_{\beta}[y_0]$. Therefore, we can apply  \textbf{(S)} with $x'=z_k,$ $\tilde q=p=p_k$ and $x=w_k$ to conclude that there exists $q'\in (g_{p_{k}} + G)^{-1}(z_k)$ such that $d(q',p_k) \leq \Theta \|z_k - w_k\|_e$. Putting $q'=p_{k+1}$ in this inequality and taking into account \eqref{defzk-slc}, \eqref{defwk-slc}, the first inequality in $(v)$, \eqref{eq:ine.ip.-sl} and (A2), we get
\begin{align*}
	d(p_{k+1},p_k) &\leq \Theta \|z_k - w_k\|_e \\
	&\le \Theta \left( \|-f(p_k) +f(p_{k-1})+{\cal D}f(p_{k-1})[\exp^{-1}_{p_{k-1}} p_k] \|_e + \|u_k \|_e + \| u_{k-1}\|_e \right) \nonumber\\
	&\leq \Theta ( \epsilon + \iota) d(p_k, p_{k-1}) 
	= \hat{\alpha} d(p_k, p_{k-1}) \leq \hat{\alpha}^k \Theta \|y_0\|_e\le \hat{\alpha}^k \Theta(1+\iota) \|y_0\|_e.
\end{align*}
In view of this and (A1), we also have
        \begin{align*}
		d(p_{k+1}, p_0)&\leq d(p_{k+1}, p_k) + d(p_{k}, p_0)\nonumber\\
		&\leq \hat{\alpha}^k\Theta(1+\iota)\|y_0\|_e+\frac{1-\hat{\alpha}^{k}}{1-\hat{\alpha}}\Theta(1+\iota)\|y_0\|_e=\frac{1-\hat{\alpha}^{k+1}}{1-\hat{\alpha}}\Theta(1+\iota)\|y_0\|_e.
	\end{align*}
	Furthermore, it comes from $p_{k+1} \in (g_{p_k} + G)^{-1}(z_k)$ that
	$$
	u_k  \in   f(p_k) +
	{\cal D}f(p_k) [ \exp_{p_k}^{-1} p_{k+1}]  + F(p_{k+1}).
	$$
	Overall, we can conclude that (A1), (A2) and (A3) hold for all $k\geq 1.$ 

Now our focus is to show the convergence of $\{p_k\}.$ Using (A1) and (A2) in a simple induction procedure, we can show that
$$
d(p_{m}, p_n)\leq \hat{\alpha}^m \frac{1-\hat{\alpha}^{n-m}}{1-\hat{\alpha}}\Theta(1+\iota)\|y_0\|_e
$$
holds for all $m<n$, $m,n\in \mathbb{N}$. Hence, since $\hat{\alpha}\in (0,1)$,  $\{p_k\}$ is a Cauchy sequence, and consequently, there exists $\bar{p}\in {\cal M}$ such that $\lim_{k \to \infty} p_k=\bar{p}.$ From (A1) and \eqref{eq:y0.bound.s-l}, it follows that $p_k\in \B_{\delta}(p_0)$ for all $\mathbb{N}$, which implies that $\bar{p}\in \B_{\delta}[p_0]\subset \Omega$. As $f$ is continuous and $F$ has closed graph, by letting $k \to +\infty$ in \eqref{eq:ine.ip.-sl} and (A3), we conclude that $\bar{p}$ is a solution of \eqref{eq:mainproblem}. Finally, using (A2), we get
\begin{align*}
	d(p_k, \bar{p})&= \lim_{m \to \infty} d(p_{k}, p_{k+m})\leq \lim_{m \to \infty}\sum_{i=k}^{k-1+m}d(p_i, p_{i+1})\\
	&\leq \lim_{m \to \infty}\sum_{i=k}^{k-1+m}\hat{\alpha}^i\Theta(1+\iota)\|y_0\|_e= \lim_{m \to \infty}\frac{\hat{\alpha}^k(1-\hat{\alpha}^m)}{1-\hat{\alpha}}	\Theta(1+\iota)\|y_0\|_e
	= \frac{\hat{\alpha}^k}{1-\hat{\alpha}}	\Theta(1+\iota)\|y_0\|_e,
\end{align*}
for all $k\in \mathbb{N}$.
\end{proof}

We conclude this section with some remarks about Theorem~\ref{teo:semilocal}. First, it is a result to guarantee the existence of a neighborhood which the inexact Newton's method is well-defined and, hence its convergence to a solution of \eqref{eq:mainproblem}. Thus, in practice, it can be hard to evaluate all the parameters in itens $(i)-(v)$. Second, it states that if the initial point $p_0\in {\cal M}$ satisfies the inclusion $y_0\in (f+F)(p_0)$ where $y_0\in B_{\chi}(0)$ and 
\[
0<\chi:=\min\left\{ \frac{\beta}{\Theta(1+\iota)}, \frac{\beta}{1+\iota},   \frac{\beta( 1 - \hat{\alpha} )}{\hat{\alpha} - \hat{\alpha}^2 + 1 },   \frac{\delta( 1 - \hat{\alpha} )}{\Theta}  \right\}
\]
then the inexact Newton's method  finds a solution of \eqref{eq:mainproblem} in  $B_{\tilde \chi}(p_0)$, where
\[
\tilde \chi=\frac{\Theta(1+\iota)}{1-\hat{\alpha}}\|y_0\|_e.
\]
We can  interpret this remark as follows: although we cannot evaluate $\chi$ in practice, if we  choose $y_0$ close to $0$ then the proposed method will work. Third, the assumption $\|u_0\|\le \iota\|y_0\|_e$ in $(iii)$ ensures that the first call (and the subsequent calls) to the inexact Newton's method for solving \eqref{eq:mainproblem} address the inexactness. Fourth, conditions in (i)-(ii) are usual in the context of metric regularity, and both inequalities in 
$(v)$ are related to the smoothness assumption of $f$ and the continuity of the exponential map. Thus, if $\epsilon$ and $\eta$ are small enough, this assumption holds.

\section{Application }\label{sec:rpscandeg}

In this section, we investigate three well-known problems that can be formulated as generalized equations on Riemannian manifolds.

\begin{example}[System of Inequalities and Equalities]
Consider the generalized equation \eqref{eq:mainproblem} with $F \equiv -K$, where $K\subset\mathbb{R}^m$ is the fixed cone 
$$
\mathbb{R}_{-}^{s} \times \{0\}^{m-s} \coloneqq \{x\in\mathbb{R}^m\colon x_i\leq 0 \,\, \mbox{for} \,\, i=1,\ldots,s \,\, \mbox{and} \,\, x_i= 0 \,\, \mbox{for} \,\, i=s+1,\ldots,m\}.
$$
It is easy to see that this generalized equation is equivalent to the following system of equalities and inequalities: 
\begin{equation}\label{sistem}
\begin{aligned}
f_i(p) & \leq 0 \quad \text{for } i = 1, \ldots, s, \\
f_i(p) & = 0 \quad \text{for } i = s+1, \ldots, m.
\end{aligned}
\end{equation}
Let $\bar{p}$ solve \eqref{sistem} and let each $f_i$ be continuously differentiable around $\bar{p}$ for all $i=1,\ldots,m$. As defined in \cite[Definition 3.12]{BergmannHerzog2019}, the Mangasarian--Fromovitz condition for system \eqref{sistem} is as follows: 
\begin{equation}\label{MFCQ}
\exists v \in {\cal T}_{\bar{p}}{\cal M} \text{ such that} \quad
\begin{cases}
\langle \grad f_i(\bar{p}), v \rangle < 0 & \quad \text{for } i \in \{1, \ldots, s\} \text{ with } f_i(\bar{p}) = 0, \\
\langle \grad f_i(\bar{p}), v \rangle = 0 & \quad \text{for } i \in \{s+1, \ldots, m\}.
\end{cases}
\end{equation}
After making simple adaptations of \cite[Example 4D.3]{DontchevRockafellar2014} to the Riemannian context, we can assert that condition \eqref{MFCQ} is equivalent to the Aubin property of \( (-f + K)^{-1} \) at \( \bar{x} \) for \( \bar{p} \), which, in turn, by Theorem \ref{teo:Aubin_MR}, is equivalent to the metric regularity of \( -f + K \) at \( \bar{p} \) for \( \bar{x} \).
\end{example}

The following proposition serves as a prerequisite for discussing the last two examples in this section. This result establishes an equivalence between problems  \eqref{eq:mainproblem} and \eqref{eq:VF}.
\begin{proposition}\label{prop:reqeezconhadsc}
Let \(\Omega \subseteq \mathcal{M}\) and let \(\{E_1, \ldots, E_n\}\) be a basis for \(\mathcal{X}(\Omega)\). Suppose \(V \colon \Omega \to \mathcal{T}\mathcal{M}\) is a single-valued vector field, and \(Z \colon \Omega \rightrightarrows \mathcal{T}\mathcal{M}\) is a set-valued vector field. Then, a point \(\bar{p}\) is a solution to \eqref{eq:VF} if and only if it is a solution to the generalized equation \eqref{eq:mainproblem} with \(f \colon \Omega \to \mathbb{R}^n\) and \(F \colon \Omega \rightrightarrows \mathbb{R}^n\) defined, respectively, by 
\begin{equation}\label{def:fandF.apsec}
f(p) \coloneqq (\langle V(p), E_1(p) \rangle, \ldots, \langle V(p), E_n(p) \rangle) \quad \text{and} \quad F(p) \coloneqq \bigcup_{v \in Z(p)} (\langle v, E_1(p) \rangle, \ldots, \langle v, E_n(p) \rangle).
\end{equation}
\end{proposition}
\begin{proof}
If \(\bar{p}\) is a solution of \eqref{eq:VF}, then there exists \(\bar{v} \in Z(\bar{p})\) such that \(V(\bar{p}) + \bar{v} = 0_{\bar{p}}\).
Consequently,
\[
\langle V(\bar{p}), E_i(\bar{p}) \rangle + \langle \bar{v}, E_i(\bar{p}) \rangle = \langle V(\bar{p}) + \bar{v}, E_i(\bar{p}) \rangle = 0 \quad \text{for all } i = 1, \ldots, n.
\]
Using the definitions of \(f\) and \(F\) in \eqref{def:fandF.apsec}, we find that \(\bar{p}\) satisfies \eqref{eq:mainproblem}. Conversely, if \(\bar{p}\) is a solution to \eqref{eq:mainproblem} with \(f\) and \(F\) defined in \eqref{def:fandF.apsec}, then there exists \(\bar{v} \in Z(\bar{p})\) such that
\[
(\langle V(\bar{p}), E_1(\bar{p}) \rangle, \ldots, \langle V(\bar{p}), E_n(\bar{p}) \rangle) + (\langle \bar{v}, E_1(\bar{p}) \rangle, \ldots, \langle \bar{v}, E_n(\bar{p}) \rangle) = 0.
\]
Since \(\{E_1(\bar{p}), \ldots, E_n(\bar{p})\}\) forms a basis for \(\mathcal{T}_{\bar{p}}\mathcal{M}\), it follows that \(\langle V(\bar{p}) + \bar{v}, v \rangle = 0\) for all \(v \in \mathcal{T}_{\bar{p}}\mathcal{M}\).
Thus, \(0_{\bar{p}} = V(\bar{p}) + \bar{v}\), implying that \(\bar{p}\) is a solution to \eqref{eq:VF}.
\end{proof}

In the following example, we discuss the variational inequality problem proposed in \cite{ShuLongLi.ChongLi.YeongChengLiou.JenChihYao.2009}, which extends the problem introduced in \cite{nemeth2003}.

\begin{example}[Variational Inequality Problem]
    Let \(V \colon \Omega\subseteq {\cal M} \to \mathcal{T}\mathcal{M}\) be a single-valued vector field. Consider the variational inequality problem 
    \begin{equation}\label{eq:prob.des.vari}
        p \in \Omega, \quad \langle V(p), \dot{\gamma}(0) \rangle \geq 0 \quad \text{for all } \gamma \in \Gamma^{\Omega}_{p,q},
    \end{equation}
    where \(\Gamma^{\Omega}_{p,q}\) is the set of all geodesics \(\gamma \colon [0, 1] \to \mathcal{M}\) satisfying \(\gamma(0)=p\), \(\gamma(1)=q\), and \(\gamma(t) \in \Omega\) for all \(t \in [0,1]\). Since the normal cone associated with \(\Omega\) is the set-valued vector field \(N_{\Omega} \colon \Omega \rightrightarrows \mathcal{T}\mathcal{M}\) defined by
\begin{equation*}\label{eq:cone.normal}
    N_{\Omega}(p) \coloneqq \left\{ v \in \mathcal{T}_p \mathcal{M} \colon \langle v, \dot{\gamma}(0) \rangle \leq 0 \text{ for each } \gamma \in \Gamma^{\Omega}_{p,q}\right\},
\end{equation*}
for all \(p \in \Omega\) (see \cite{LiChong-YaoJenChih.2012}), the problem \eqref{eq:prob.des.vari} is equivalent to
\[
    p \in \Omega, \quad V(p) + N_{\Omega}(p) \ni 0_p,
\]
which, in turn, by Proposition \ref{prop:reqeezconhadsc}, is equivalent to generalized equation \eqref{eq:mainproblem} with \(f \colon \Omega \to \mathbb{R}^n\) and \(F \colon \Omega \rightrightarrows \mathbb{R}^n\) as in \eqref{def:fandF.apsec}, with \(Z = N_{\Omega}\) and \(\{E_1, \ldots, E_n\}\) being any basis for \(\mathcal{X}(\Omega)\).
\end{example}

The following example proposes an approach based on the Riemannian extension of the analysis presented in \cite{izmailov2010inexact}.

\begin{example}[KKT Conditions]\label{ex:kktconditions}
    Consider the constrained nonlinear optimization problem on $ {\cal M}$:
    \begin{align}
        \text{minimize } & \, {\bf f}(p)  \label{pro:kkt1l}\\
        \text{subject to } & \,  {\bf g}(p) \leq 0, \,\, {\bf h}(p) = 0,  \label{pro:kkt2l}
    \end{align}
    where ${\bf f} \colon {\cal M} \to \mathbb{R}$, ${\bf g} \coloneqq ({\bf g}_1, \ldots, {\bf g}_{m_1}) \colon {\cal M} \to \mathbb{R}^{m_1}$, and ${\bf h} \coloneqq ({\bf h}_1, \ldots, {\bf h}_{m_2}) \colon {\cal M} \to \mathbb{R}^{m_2}$ are assumed to be continuously differentiable functions. The constraint ${\bf g}(p) \leq 0$ means that ${\bf g}_i(p) \leq 0$ for all $i = 1, \ldots, m_1$. The Lagrangian function ${\cal L}\colon {\cal M} \times \mathbb{R}^{m_1} \times \mathbb{R}^{m_2} \to \mathbb{R}$ is given by
    \begin{equation*}\label{eq:lagrangiana}
        \mathcal{L}(p, \mu, \lambda) := {\bf f}(p) + \sum_{i=1}^{m_1} \mu_i {\bf g}_i(p) + \sum_{i=1}^{m_2} \lambda_i {\bf h}_i(p),
    \end{equation*}
    where $\mu \coloneqq (\mu_1, \ldots, \mu_{m_1}) \in \mathbb{R}^{m_1}$ and $\lambda \coloneqq (\lambda_1, \ldots, \lambda_{m_2}) \in \mathbb{R}^{m_2}$. For each $(\mu,\lambda) \in \mathbb{R}^{m_1} \times \mathbb{R}^{m_2}$, consider the function ${\cal L}_{\mu,\lambda} \colon {\cal M} \to \mathbb{R}$ defined by ${\cal L}_{\mu,\lambda}(p) = \mathcal{L}(p, \mu, \lambda)$ for all $p \in {\cal M}$. Based on \cite{BergmannHerzog2019}, we assert that the KKT conditions for \eqref{pro:kkt1l}--\eqref{pro:kkt2l} are:
    \begin{align}
        \grad \mathcal{L}_{\mu, \lambda}(p) = \grad {\bf f}(p) + \sum_{i=1}^{m_1} \mu_i \grad {\bf g}_i(p) + \sum_{i=1}^{m_2} \lambda_i \grad {\bf h}_i(p) = 0_p, \label{pro:condkkt1l}\\    
        \mu \geq 0, \quad {\bf g}(p) \leq 0, \quad 
        \sum_{i=1}^{m_1} \mu_i {\bf g}_i(p) = 0, \label{pro:condkkt2}\\
        {\bf h}(p) = 0.          \label{pro:condkkt3}
    \end{align}
    Let \(\widetilde{\mathcal{M}} \coloneqq \mathcal{M} \times \mathbb{R}^{m_1} \times \mathbb{R}^{m_2}\) and consider the vector field \(V \colon \widetilde{\mathcal{M}} \to \mathcal{T}\widetilde{\mathcal{M}} \equiv \mathcal{T}\mathcal{M} \times \mathbb{R}^{m_1} \times \mathbb{R}^{m_2}\) defined by
\begin{equation*}\label{appkktV}
    V(p, \mu, \lambda) = (\grad \mathcal{L}_{\mu, \lambda}(p), {\bf g}(p), {\bf h}(p)),
\end{equation*}
and the set-valued vector field \(Z \colon \widetilde{\mathcal{M}} \rightrightarrows \mathcal{T}\widetilde{\mathcal{M}}\) defined by
\begin{equation*}\label{appkktZ}
    Z(p, \mu, \lambda) = \left\{ 
    \begin{array}{lll}
        \{0_p\} \times \left\{ y \in \mathbb{R}_+^{m_1} \colon \sum_{i=1}^{m_1} \mu_i y_i = 0 \right\} \times \{0\}, & \text{if } \mu \geq 0; \\
     \,   \emptyset, & \text{otherwise},
    \end{array} 
    \right.
\end{equation*}
where \(\mathbb{R}_+^{m_1}\) denotes the set of vectors in \(\mathbb{R}^{m_1}\) with nonnegative coordinates. Note that KKT system \eqref{pro:condkkt1l}--\eqref{pro:condkkt2}--\eqref{pro:condkkt3} is equivalent to the problem 
\begin{equation}\label{eq:prob.pczkktsecdc}
     (p, \mu, \lambda) \in \widetilde{\mathcal{M}}, \quad V(p, \mu, \lambda) + Z(p, \mu, \lambda) \ni (0_p, 0, 0).
\end{equation}
Let \(\{E_{1}, \ldots, E_{n}\}\) be a basis for \(\mathcal{X}(\mathcal{M})\) and \(\{E_{n+1}, \ldots, E_{n+m_1+m_2}\}\) be the canonical basis for \(\mathbb{R}^{m_1} \times \mathbb{R}^{m_2}\). Then, \(\{E_1, \ldots, E_{n+m_1+m_2}\}\) forms a basis for \(\mathcal{X}(\widetilde{\mathcal{M}})\) and, by Proposition \ref{prop:reqeezconhadsc}, \eqref{eq:prob.pczkktsecdc} is equivalent to the generalized equation
\begin{equation*}\label{eq:eggexapp}
    (p, \mu, \lambda) \in \widetilde{\mathcal{M}}, \quad f(p, \mu, \lambda) + F(p, \mu, \lambda) \ni 0,
\end{equation*}
where \(f \colon \widetilde{\mathcal{M}} \to \mathbb{R}^n \times \mathbb{R}^{m_1} \times \mathbb{R}^{m_2}\) and \(F \colon \widetilde{\mathcal{M}} \rightrightarrows \mathbb{R}^n \times \mathbb{R}^{m_1} \times \mathbb{R}^{m_2}\) are defined, respectively, by
\begin{equation*}\label{def:fandF.apsecfkkt}
    (f (p, \mu, \lambda))_i = \left\{ 
    \begin{array}{lll}
        \langle \grad \mathcal{L}_{\mu, \lambda}(p), E_i(p) \rangle, & \text{if } i=1,\ldots,n; \\
        ({\bf g}(p))_{i-n}, & \text{if } i=n+1,\ldots,n+m_1; \\
        ({\bf h}(p))_{i-n-m_1}, & \text{if } i=n+m_1+1,\ldots,n+m_1+m_2;
    \end{array} 
    \right.
\end{equation*}
and
\begin{equation*}\label{def:fandF.apsecFkkt}
    F(p, \mu, \lambda) = \left\{ 
    \begin{array}{lll}
        \{0\} \times \left\{ y \in \mathbb{R}_+^{m_1} \colon \sum_{i=1}^{m_1} \mu_i y_i = 0 \right\} \times \{0\}, & \text{if } \mu \geq 0; \\
       \,  \emptyset, & \text{otherwise}.
    \end{array} 
    \right.
\end{equation*}
\end{example}



\section{Numerical Example}\label{sec:numerical_example}

In this section, we present a numerical example based on a generalized equation derived from Example \ref{ex:kktconditions}, and solve it using the inexact Newton method described in \eqref{eq:INMtp}. All computations were performed on a MacBook Pro running macOS Sonoma 14.5, equipped with 16 GB RAM, an Apple M1 Pro CPU, and MATLAB R2022a. To ensure reproducibility, 
we fixed the randomness using MATLAB's built-in function \textit{rng(2024)}.

Here, we consider the Riemannian constrained optimization problem on $ {\cal M}$:
\begin{align}
    \text{minimize } \,\, & \, {\bf f}(p) \coloneqq \frac{1}{N}\sum_{i=1}^{N}d^{2}(p, p^i) \label{pro:kkt1lapp} \\
    \text{subject to } & \, {\bf g}(p) \coloneqq d^{2}(p, \tilde{p}) - r^2 \leq 0, \label{pro:kkt2lapp} 
\end{align}
where \( r > 0 \) and \( p^1, \ldots, p^N, \tilde{p} \in \cal{M} \) are chosen such that \( r < r_{inj}(\tilde{p}) \) and \( p^1, \ldots, p^N \in {\cal B}_r(\tilde{p}) \). We will use the fact that
\begin{equation}\label{secap.grad.fg}
\grad {\bf f}(p) = -\frac{2}{N}\sum_{i=1}^{N}\exp^{-1}_{p}p^{i} \quad \text{and} \quad \grad {\bf g}(p) = -2\exp^{-1}_{p}\tilde{p}, \quad \forall p \in {\cal B}_r(\tilde{p}).
\end{equation}
The problem defined by \eqref{pro:kkt1lapp}--\eqref{pro:kkt2lapp} represents a constrained version of the Riemannian center of mass, also known as the (Riemannian) mean, as proposed in \cite{BergmannHerzog2019}. This problem was first introduced in \cite{karcher1977riemannian} and has been extensively studied in recent literature (see, for example, \cite{bini2013computing,ferreira2019gradient,weber2023riemannian}).

Particularly, we are interested in problem \eqref{pro:kkt1lapp}--\eqref{pro:kkt2lapp} with $\mathcal{M} = \mathbb{S}^3 \coloneqq \{ p \in \mathbb{R}^4 \ |\ \|p\|_e = 1 \}$, equipped with the metric of the ambient space $\mathbb{R}^4$. According to \cite{hesselholtvector}, the skew-symmetric matrices

\[
M_1 = \begin{bmatrix}
0 & -1 & 0 & 0\\
1 & 0 & 0 & 0\\
0 & 0 & 0 & -1\\
0 & 0 & 1 & 0
\end{bmatrix}, \quad
M_2 = \begin{bmatrix}
0 & 0 & -1 & 0\\
0 & 0 & 0 & 1\\
1 & 0 & 0 & 0\\
0 & -1 & 0 & 0
\end{bmatrix}, \quad
M_3 = \begin{bmatrix}
0 & 0 & 0 & -1\\
0 & 0 & -1 & 0\\
0 & 1 & 0 & 0\\
1 & 0 & 0 & 0
\end{bmatrix},
\]
induce an orthonormal basis of vector fields $\{ E_1, E_2, E_3 \}$ on $\mathbb{S}^3$, defined by $E_i(p) = M_i p$ (standard matrix-vector product) for all $p \in \mathbb{S}^3$ and $i = 1, 2, 3$. By defining $\mathcal{L}_{\mu} \colon \mathbb{S}^3 \to \mathbb{R}$ by 
\begin{equation*}
    \mathcal{L}_{\mu}(p) \coloneqq {\bf f}(p) + \mu {\bf g}(p),
\end{equation*}
it follows from Example \ref{ex:kktconditions} that the KKT conditions for \eqref{pro:kkt1lapp}--\eqref{pro:kkt2lapp} are:
\begin{align}
    \grad \mathcal{L}_{\mu}(p) = \grad {\bf f}(p) + \mu \grad {\bf g}(p) &= 0_p,\label{kkt1}
    \\
    \mu \geq 0, \quad {\bf g}(p) \leq 0, \quad \mu {\bf g}(p) &= 0,  \label{kkt2}
\end{align}
which is equivalent to the generalized equation
\begin{equation}\label{GEquation}
(p, \mu) \in \mathbb{S}^3 \times \mathbb{R}, \quad f(p, \mu) + F(p, \mu) \ni 0,
\end{equation}
where \(f \colon \mathbb{S}^3 \times \mathbb{R} \to \mathbb{R}^3 \times \mathbb{R}\) and \(F \colon \mathbb{S}^3 \times \mathbb{R} \to \mathbb{R}^3 \times \mathbb{R}\) are defined, respectively, by
\begin{equation}\label{V_formulation}
f(p, \mu) = (\langle \grad \mathcal{L}_{\mu}(p), E_1(p) \rangle, \langle \grad \mathcal{L}_{\mu}(p), E_2(p) \rangle, \langle \grad \mathcal{L}_{\mu}(p), E_3(p) \rangle, {\bf g}(p))
\end{equation}
and
\begin{equation}\label{Z_formulation}
F(p, \mu) = \left\{ 
\begin{array}{ll}
    \{0\} \times \left\{ y \in \mathbb{R}_+ \colon \mu y = 0 \right\}, & \text{if } \mu \geq 0, \\
    \emptyset, & \text{otherwise},
\end{array} 
\right.
\end{equation}
for all $(p, \mu) \in \mathbb{S}^3 \times \mathbb{R}$.

To apply the inexact Newton method in \eqref{eq:INMtp} to solve \eqref{GEquation}, the subproblem in each iteration \(k\) involves computing \((v_k, \nu_k) \in \mathcal{T}_{(p_k, \mu_k)}(\mathbb{S}^3 \times \mathbb{R}) \equiv \mathcal{T}_{p_k} \mathbb{S}^3 \times \mathbb{R}\) such that
\begin{equation*}\label{EINN}
\left( f(p_k, \mu_k) + \mathcal{D}f(p_k, \mu_k)[ (v_k, \nu_k) ] + F(\exp_{p_k} v_k, \mu_k + \nu_k) \right) \cap R_k(p_k, \mu_k) \neq \varnothing.
\end{equation*}
To accomplish this, select \( u_k \in R_k(p_k, \mu_k) \) (we use \( u_k = \left[ 1/(10^k) , 1/(10^k) , 1/(10^k) , 1/(10^k) \right]^{t} \), where \( t \) denotes the transpose operation, for all \( k \)), and then solve the following optimization problem:
\begin{align}
    \text{ minimize } \,\,& \,\, \frac{1}{2}\|z + f(p_k, \mu_k) + \mathcal{D}f(p_k, \mu_k)[(v, \nu)] - u_k \|_e^2  \label{pro:kkt1lapp.subp1} \\ 
    \text{subject to } & \, (v, \nu) \in \mathcal{T}_{p_k} \mathbb{S}^3 \times \mathbb{R}, \quad z = (z_1, z_2, z_3, z_4) \in F(\exp_{p_k} v, \mu_k + \nu). \label{pro:kkt2lapp.subp1}
\end{align}
Based on the definition of \( F \) in \eqref{Z_formulation} and the fact that every \( v \in \mathcal{T}_{p_k} \mathbb{S}^3 \) can be expressed as a linear combination of \( E_{1}(p_{k}) \), \( E_{2}(p_{k}) \), and \( E_{3}(p_{k}) \), we can find a solution to \eqref{pro:kkt1lapp.subp1}--\eqref{pro:kkt2lapp.subp1} by solving the Euclidean quadratic constraint problem:
\begin{align}
    \text{minimize } \,\, & \, \frac{1}{2}\|z + f(p_k, \mu_k) + \mathcal{D}f(p_k, \mu_k)[(\alpha_1E_{1}(p_k) + \alpha_2E_{2}(p_k) + \alpha_3E_{3}(p_k), \nu)] - u_k \|_e^2  \label{pro:kkt1lapp.subp2} \\ 
    \text{subject to } & \, (\alpha_1, \alpha_2, \alpha_3) \in \mathbb{R}^3, \,\, z_1 = z_2 = z_3 = 0, \,\, z_4 \in \mathbb{R}^{+}, \,\, \mu_k + \nu \in \mathbb{R}^{+}, \,\, z_4(\mu_k + \nu) = 0. \label{pro:kkt2lapp.subp2}
\end{align}
Thus, the iteration $k+1$ is obtained as follows:
\[
(p_{k+1}, \mu_{k+1}) = (\exp_{p_k}(\alpha^k_1E_{1}(p_k) + \alpha^k_2E_{2}(p_k) + \alpha^k_3E_{3}(p_k)), \mu_k + \nu_k)
\]
where $\alpha^k_1, \alpha^k_2, \alpha^k_3, \nu_k$ are a solution for \eqref{pro:kkt1lapp.subp2}--\eqref{pro:kkt2lapp.subp2}.

For the implementation of our algorithm, we use the following expressions derived from \eqref{secap.grad.fg} and \eqref{V_formulation}:
\begin{equation*}
f(p_k, \mu_k) = 
\begin{bmatrix}
\left\langle -\frac{2}{N} \sum_{i=1}^{N} \exp^{-1}_{p_k} p^i - 2\mu_k \exp^{-1}_{p_k} \tilde{p}, E_1(p_k) \right\rangle \\
\left\langle -\frac{2}{N} \sum_{i=1}^{N} \exp^{-1}_{p_k} p^i - 2\mu_k \exp^{-1}_{p_k} \tilde{p}, E_2(p_k) \right\rangle \\
\left\langle -\frac{2}{N} \sum_{i=1}^{N} \exp^{-1}_{p_k} p^i - 2\mu_k \exp^{-1}_{p_k} \tilde{p}, E_3(p_k) \right\rangle \\
d^2(p_k, \tilde{p}) - r^2
\end{bmatrix}
\end{equation*}
and
\begin{equation*}
\mathcal{D}f(p_k, \mu_k)[(v, \nu)] = 
\begin{bmatrix}
\langle \grad_p f_1(p_k, \mu_k), v \rangle - 2 \langle \log_{p_k} \tilde{p}, E_1(p_k) \rangle \nu \\
\langle \grad_p f_2(p_k, \mu_k), v \rangle - 2 \langle \log_{p_k} \tilde{p}, E_2(p_k) \rangle \nu \\
\langle \grad_p f_3(p_k, \mu_k), v \rangle - 2 \langle \log_{p_k} \tilde{p}, E_3(p_k) \rangle \nu \\
-2 \langle \log_{p_k} \tilde{p}, v \rangle
\end{bmatrix}, \quad  (v, \nu) \in \mathcal{T}_{p_k} \mathbb{S}^3 \times \mathbb{R},
\end{equation*}
where, for each \( j \in \{1, 2, 3\} \), \( \grad_p f_j(p_k, \mu_k) \) denotes the Riemannian gradient of \( f_j(\cdot, \mu_k) \) at \( p_k \), which is the projection of the Euclidean gradient of \( f_j(\cdot, \mu_k) \) at \( p_k \) onto \( \mathcal{T}_{p_k} \mathbb{S}^3 \). This projection can be obtained by multiplying the vector by the matrix $ \id_{4 \times 4} - p_k (p_k)^t $, where  \( \id_{4 \times 4} \) denotes the identity matrix of dimension \( 4 \times 4 \). Additionally, we need to know the following expressions:
\begin{equation*}
\exp_{p}(v) \coloneqq \left\{
\begin{split}
& \cos(\|v\|_{2})p + \sin(\|v\|_{2})\frac{v}{\|v\|_{2}},\ v \in \mathcal{T}_{p} \mathbb{S}^{3} \setminus \{0\}, \\
& p, \qquad\qquad\qquad\qquad\qquad\qquad   v = 0,
\end{split} \right.
\end{equation*}
and
\begin{equation*}
\exp_{p}^{-1}(q) \coloneqq \left\{
\begin{split}
& \frac{\arccos\langle p, q \rangle}{\sqrt{1 - \langle p, q \rangle^{2}}}(I - p p^T)q,\ q \notin \{ p, -p \}, \\
& 0, \qquad\qquad\qquad\qquad\qquad \ q = p.
\end{split} \right.
\end{equation*}
These formulas on the sphere can be found in \cite{ferreira2013projections}, for example.

Now, we are prepared for the numerical implementation. We use the Matlab built-in function \textit{rand} to generate \( p^i \), normalizing them to unit vectors. Each component of \( p^i \) lies within the interval \( (0, 1) \). Specifically, we consider four cases:

\begin{itemize}
\item[\bf A1.] \( N=10 \), \( r=2 \), \( \tilde{p} = [0,0,0,1] \), and \( p_0 = p^1 \);
\item[\bf A2.] \( N=500 \), \( r=2 \), \( \tilde{p} = [0,0,0,1] \), and \( p_0 = p^1 \);
\item[\bf A3.] \( N=10 \), \( r=0.1 \), \( \tilde{p} = [0,0,0,1] \), and \( p_0 = p^1 \);
\item[\bf A4.] \( N=500 \), \( r=0.1 \), \( \tilde{p} = [0,0,0,1] \), and \( p_0 = p^1 \);
\end{itemize}
where \( p_0 \) is the initial iteration point. For the stopping criteria, we use
\begin{equation*}
\| \Phi(p_{k}, \mu_{k})  \|_{e} \leq 10^{-12} \quad \mathrm{and} \quad {\bf g}(p_k) \leq 10^{-12}
\end{equation*}
where $\Phi \coloneqq (f_1,f_2,f_3)$.

From Figure 1 and Table 1, we can claim that for the above four cases {\bf A1}-{\bf A4}, we find a solution $(p^{\star},\mu^{\star})$ for generalized equation \eqref{GEquation}. This is because for all four cases, we achieve $\|\Phi(p^{\star},\mu^{\star})\|_{e}\leq 10^{-12}$, $ {\bf g}(p^{\star})\leq 10^{-12}$ and $\mu^{\star}\geq 0$. In addition,  KKT system \eqref{kkt1}-\eqref{kkt2} is also satisfied under $(p^{\star},\mu^{\star})$, which is asserted in theory and again ensured numerically. Furthermore, looking into cases {\bf A1}-{\bf A2}, we have ${\bf g}(p^{\star})<0$, which means that $p^{\star}$ lies in the interior of the constraint region. But for cases {\bf A3}-{\bf A4}, $p^{\star}$ is on the boundary of the constraint region because ${\bf g}(p^{\star})$ is almost equal to $0$. Due to the effect of the constraint, the convergence rate of the norm of $\Phi(p_{k},\mu_{k})$ for cases {\bf A1}-{\bf A2} is stable. For cases {\bf A3}-{\bf A4}, the convergence rate of the norm of $\Phi(p_{k},\mu_{k})$ in the first several iterations is unstable but after several iterations, the convergence rate becomes stable. This coincides with the theory that the convergence rate of the inexact Newton method becomes stable when the iteration point is close to the solution.

\begin{figure}[hthp!]
\centering
\subfigure[Case {\bf A1}]{
\includegraphics[width=2.5in,height=2.0in]{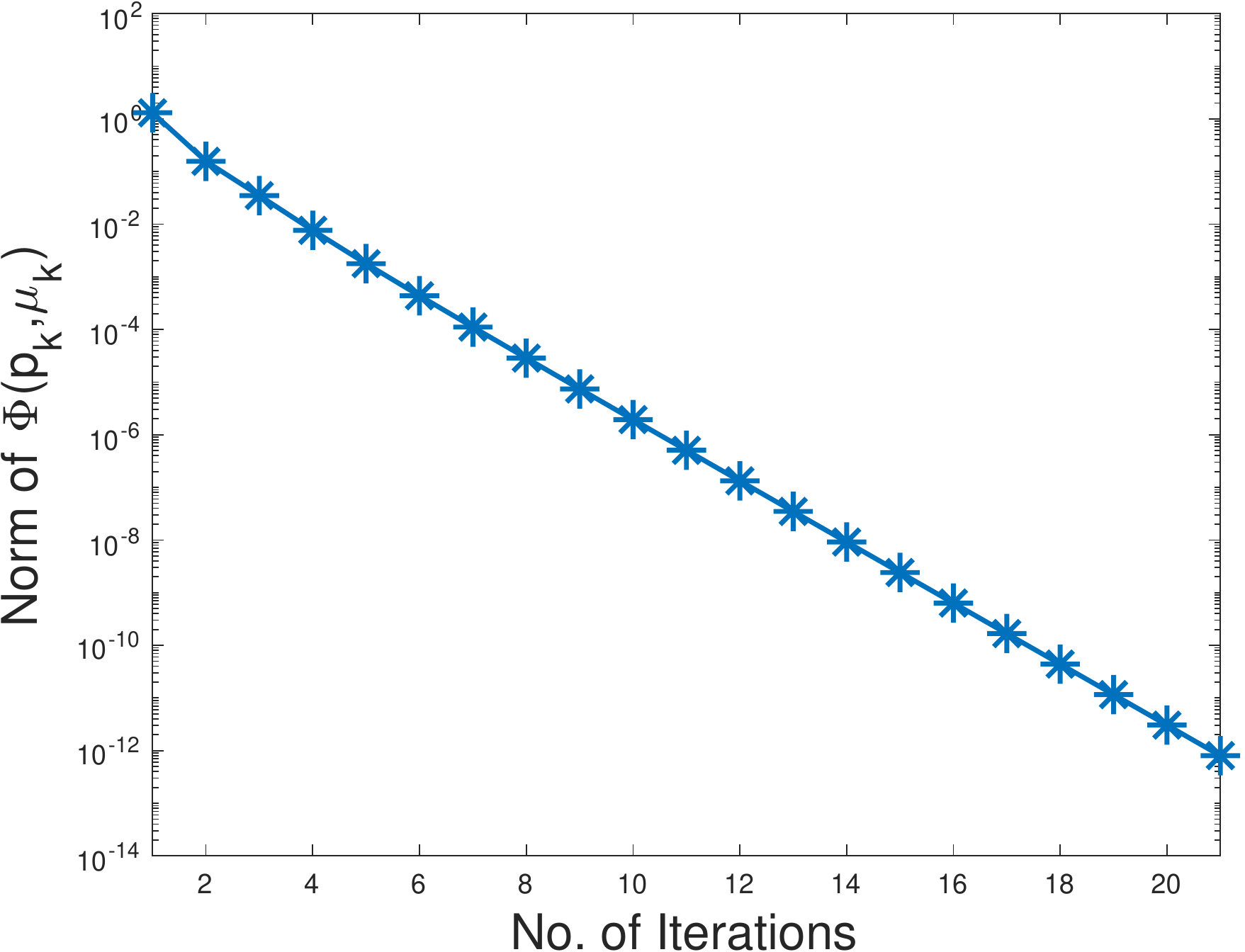}}
\subfigure[Case {\bf A2}]{
\includegraphics[width=2.5in,height=2.0in]{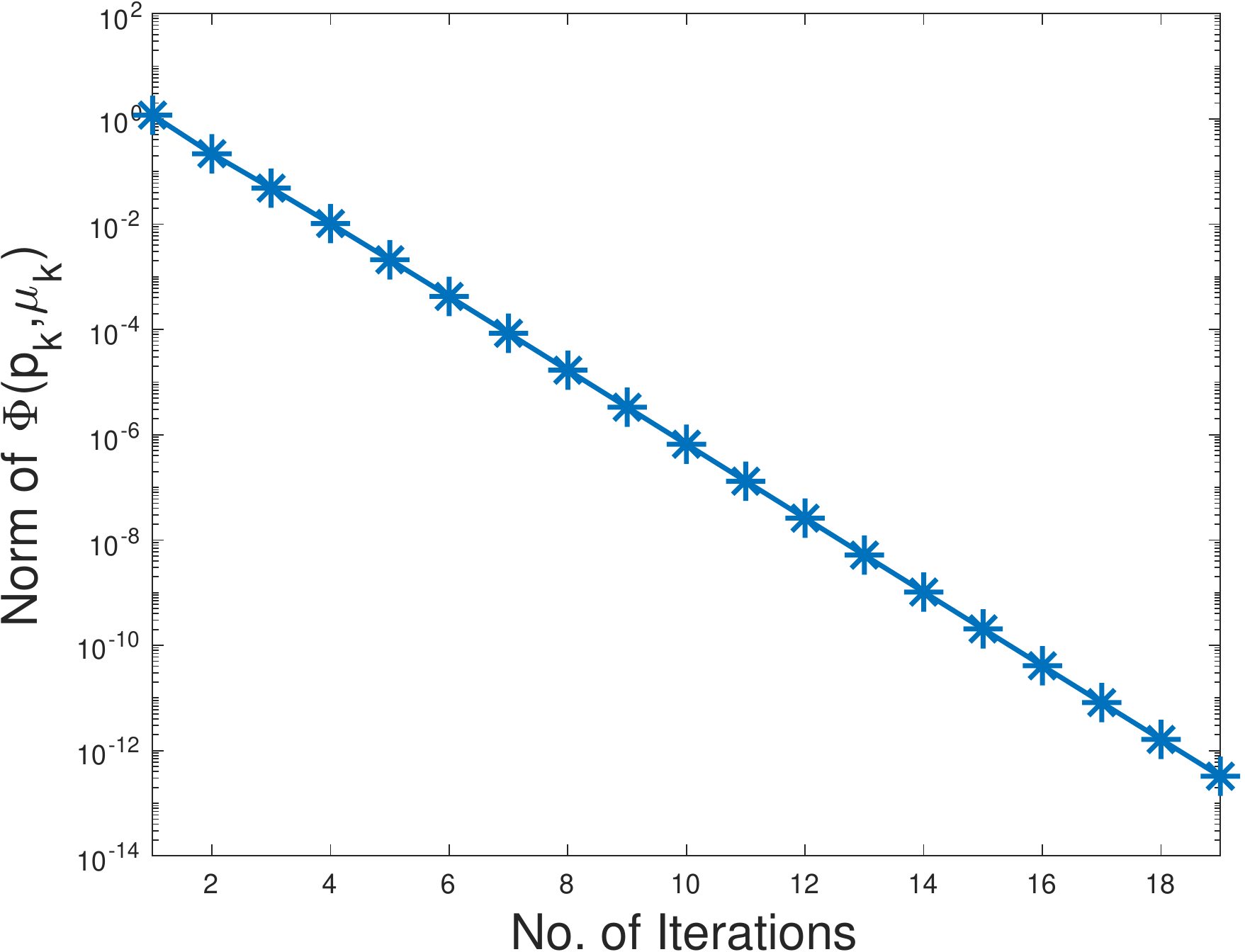}}\\
\subfigure[Case {\bf A3}]{
\includegraphics[width=2.5in,height=2.0in]{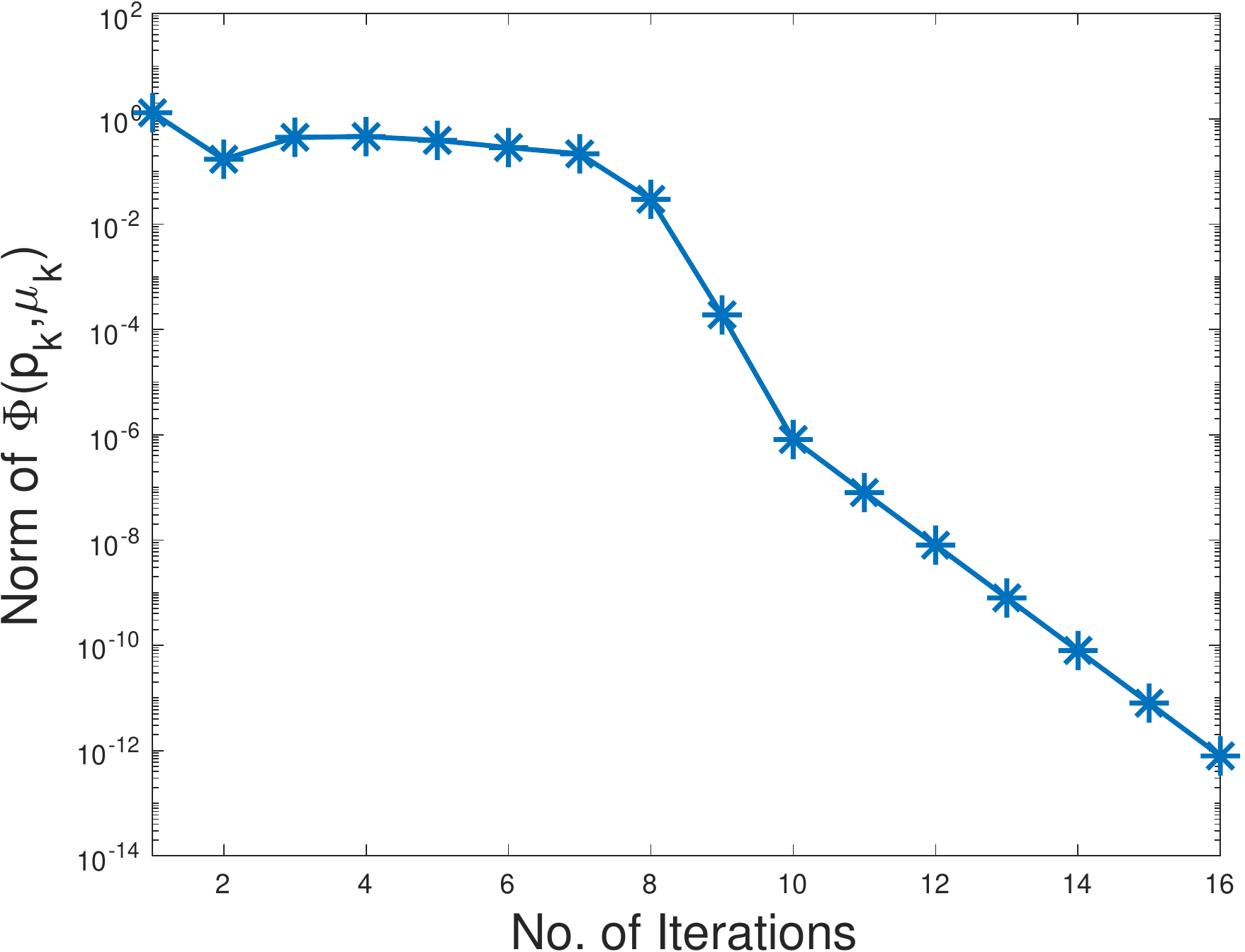}}
\subfigure[Case {\bf A4}]{
\includegraphics[width=2.5in,height=2.0in]{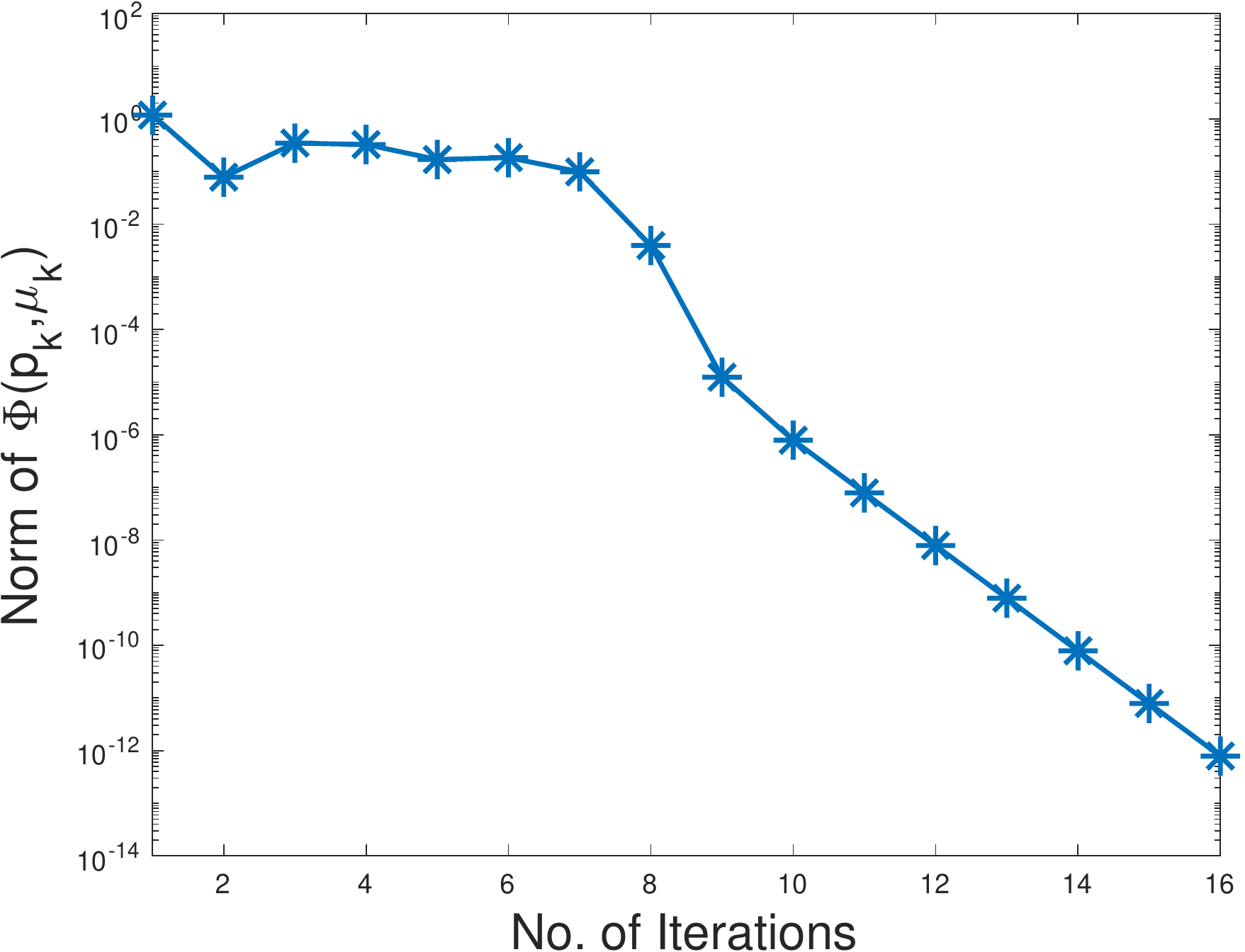}}
\caption{Norm of $\Phi(p_{k},\mu_{k})$ at each iteration for four cases {\bf A1}-{\bf A4}.}\label{figure1}
\end{figure}

\begin{table}[htbp!]
\label{table1}
\centering
\small{
\caption{The final results of four cases {\bf A1}-{\bf A4}.} 
\setlength{\tabcolsep}{1.2mm}{
\begin{tabular}{|c|c|c|c|c|c|c|}
\hline
   & $p^{\star}$  &  $\mu^{\star}$      & ${\bf g}(p^{\star})$      & $\mu {\bf g}(p^{\star})$   & $\|\grad\mathcal{L}_{\mu^{\star}}(p^{\star})\|_{p}$ & No. of Iteration  \\ \hline 
Case {\bf A1} & [0.4388,0.4862,0.5665,0.5001]    & 0      & -2.9037     & 0  &    $8.02\times 10^{-13}$ & 21    \\ \hline 
Case {\bf A2} & [0.4908,0.5100,0.4878,0.5109]    & 0      & -2.9298     & 0  &    $3.26\times 10^{-13}$ & 19    \\ \hline  
Case {\bf A3} & [0.0504,0.0566,0.0650,0.9950]    & 8.9114      & $1.05\times10^{-15}$     & $9.32\times10^{-15}$  &    $7.91\times 10^{-13}$ & 16    \\ \hline 
Case {\bf A4} & [0.0570,0.0593,0.0566,0.9950]    & 8.7870      & $1.05\times10^{-15}$     & $9.20\times10^{-15}$  &    $7.89\times 10^{-13}$ & 16    \\ \hline 
\end{tabular}}}
\end{table}

\section{Conclusion}\label{sec:Conclusions}

In this paper, we address the problem of finding solutions to generalized equations on Riemannian manifolds. If the manifold is a Euclidean space, then it is well-known that this problem encompasses several other contexts, such as standard nonlinear optimization, variational inequalities, and equilibrium problems. Here, we present a general inexact Newton method for solving generalized equations. Firstly, we discuss the metric regularity property and provide some examples of mappings satisfying this property. Secondly, we establish local convergence results, including both linear and quadratic convergence under suitable assumptions, along with a semi-local convergence result. Finally, we discuss the relationship between problems \eqref{eq:mainproblem} and \eqref{eq:VF}.

All results obtained here are derived under assumptions that are highly natural in comparison to their Euclidean counterparts, without requiring additional conditions related to the geometry of the manifold. From this perspective, considering a problem as a generalized equation may be a promising approach in the context of manifolds as well. 

As a next step, we plan to investigate how to extend the theory presented here to cases where the exponential map is replaced by a general retraction. Moreover, we intend to apply this new concept to quasi-Newton-type methods for solving problem \eqref{eq:mainproblem}.

\vspace{0.5cm}

{\bf Data Availability} Data sharing not applicable to this article as no datasets were generated or
analyzed during the current study

\vspace{0.5cm}
{\bf Declarations}

\vspace{0.5cm}

{\bf Conflict of interest} The authors declare that they have no conflict of interest

\appendix

\section{Appendix}

We begin this section by presenting two supporting lemmas. The proof of the first one can be found in \cite[Lemma A.1]{liu2020simple}. The second one is analogous to \cite[Lemma A.2]{liu2020simple}, but its proof requires some adaptations. Therefore, we provide the proof here.

\begin{lemma}\label{lem:par.trans.cont}	
	Let $( \bar{p} ,\bar{v})$ be a point on ${\cal T}{\cal M}$, and $\B_{r}(\bar{p})$ be a normal ball. Define the following vector field on $\B_{r}(\bar{p})$:
	$$
	V(p)= P_{\bar{p}p}\bar{v},\qquad p\in \B_{r}(\bar{p}).
	$$ 
	Then, $V$ is a smooth vector field on $\B_{r}(\bar{p})$.
\end{lemma}

\begin{lemma}\label{lem:bounddiff}
	Let $f\colon {\cal M} \to \mathbb{R}^m$ be a continuously differentiable function at $\bar{p}$  and $\B_{r}(\bar{p})$ a normal ball. Then, for each $\epsilon>0$,   there exists $\delta \in (0,r) $ such that 
	\begin{equation*}\label{eq:bounddiff}
		\| {\cal D}f(p)P_{ \bar{p}p}  -  {\cal D}f(\bar{p}) \|_{map} \leq \epsilon \quad \mbox{for all }   p \in \B_{\delta}(\bar{p}).
	\end{equation*}
\end{lemma}
\begin{proof}
    It follows from \eqref{eq:diff.grad. multi} and parallel transport properties that
	\begin{align*}
		{\cal D}f(p) [P_{ \bar{p}p}v] &= ({\langle}  \grad f_1(p)  ,   P_{ \bar{p}p}v  {\rangle}_p, \ldots, {\langle}  \grad f_m(p)  ,   P_{ \bar{p}p}v  {\rangle}_p),  \\
		&= ({\langle}  P_{p \bar{p}}\grad f_1(p)  ,    v  {\rangle}_{\bar{p}} , \ldots, {\langle}  P_{ p\bar{p}} \grad f_m(p)  ,   v  {\rangle}_{\bar{p}}), \quad  v\in  {\cal T}_{\bar{p}}{\cal M}, \,\,   p\in \B_{r}(\bar{p}).
	\end{align*}
	From norm properties, we have
	\begin{align*}
		\| ({\cal D}f(p) P_{ \bar{p} p}  -  {\cal D}f(\bar{p})  )[v] \|_{e}^{2} 
		& =  \sum_{i=1}^{m}
		|{\langle} P_{ p\bar{p}} \grad f_i(p)  -  \grad f_i(\bar{p})  ,   v  {\rangle}_{\bar{p}}|^2,  \\
		& \leq  \sum_{i=1}^{m}
		\| P_{ p\bar{p}} \grad f_i(p) -  \grad f_i(\bar{p})\|_{\bar{p}}^{2}  \| v \|_{\bar{p}}^{2},  \quad  v\in  {\cal T}_{\bar{p}}{\cal M}, \,\,  p\in \B_{r}(\bar{p}).
	\end{align*}
	Since ${\cal D}f(p) P_{  \bar{p} p}  -  {\cal D}f(\bar{p}) \colon {\cal T}_{\bar{p}}{\cal M} \to \mathbb{R}^m$  is a linear map, the above inequality implies that 
	\begin{align*}
		\| {\cal D}f(p) P_{  \bar{p} p}  -  {\cal D}f(\bar{p})  \|_{map}^{2} 
		& \leq  \sum_{i=1}^{m}
		\| P_{ p\bar{p}} \grad f_i(p)    - \grad f_i(\bar{p})\|_{\bar{p}}^{2}  \\
		& =  \sum_{i=1}^{m}
		\| \grad f_i(p)  -  P_{ \bar{p} p} \grad f_i(\bar{p})\|_{p}^{2}, \quad   p\in \B_{r}(\bar{p}).
	\end{align*}
	Under the assumptions of the lemma, $\grad f_i$ is a continuous vector field around $\bar{p}$ for all $i=1,\ldots,m$. Hence, using Lemma~\ref{lem:par.trans.cont} with $\bar{v} =\grad f_i(\bar{p})$ and considering the fact that the function ${\cal M} \ni p \mapsto \| \cdot \|_p$ is continuous, we obtain
	$$
	\lim_{p\to \bar{p}}\| \grad f_i(p)  -  P_{ \bar{p} p} \grad f_i(\bar{p})) \|_{p} = \| \grad f_i(\bar{p})  -  P_{ \bar{p} \bar{p}} \grad f_i(\bar{p}) \|_{\bar{p}} =0.
	$$
	Thus, it follows from the last inequality above that $\lim_{p \to \bar{p}}\| {\cal D}f(p) P_{  \bar{p} p}  -  {\cal D}f(\bar{p})  \|_{map}=0$, which is equivalent to what we want.
\end{proof}
In the following proposition, we provide a Riemannian version of the Fundamental Theorem of Calculus and two inequalities that play a crucial role throughout this paper. We claim that the inequalities in the following result do not hold for a general retraction.

\begin{proposition}\label{lem:properepsil}
	Let $f\colon {\cal M} \to \mathbb{R}^m$ be a continuously differentiable function at $\bar{p}$. Then, there exists $\delta > 0$ such that for each $p',p \in {\cal B}_{\delta}(\bar p)$, we have
	\begin{equation}\label{eq:tfc}
	f(p) - f(p') = \int_{0}^{1} {\cal D}f(\gamma(t)) [\dot{\gamma}(t)] \dt,
	\end{equation}
where $\gamma \colon [0,1] \to {\cal M}$ is a geodesic satisfying $\gamma(0)=p'$ and $\gamma(1)=p$. In particular, for each $\epsilon>0$, there exists a normal ball  $ \B_{\delta_{\epsilon}}(\bar p) $ such that 
	\begin{equation}\label{eq:properepsil}
		\| f(p) - f(\bar{p}) - {\cal D}f(\bar{p})[\exp^{-1}_{\bar p} p]\|_{e} \leq \epsilon d(p,\bar{p}), \qquad  p \in \B_{\delta_{\epsilon}}(\bar p).
	\end{equation}
	Furthermore, if ${\cal D}f$ is $L$-Lipschitz continuous around $\bar{p}$ then there exists a normal ball $\B_{\delta_L}(\bar p)$ such that
	\begin{equation}\label{eq:Df-lipschitz}
		\| f(p) - f(\bar{p}) - {\cal D}f(\bar{p}) [\exp^{-1}_{\bar p} p] \|_e \leq L d^2(p,\bar{p}), 
		\qquad  p \in \B_{\delta_L}(\bar p).
	\end{equation}
\end{proposition}
\begin{proof}
	Choose $\delta>0$ such that $f$ is continuously differentiable on ${\cal B}_{\delta}(\bar p)$. Pick $p',p \in {\cal B}_{\delta}(\bar p)$ and  consider a geodesic  $\gamma   \colon [0,1]\to {\cal M}$ connecting $p' = \gamma(0)$ to $p = \gamma(1)$. Using \cite[Proposition 2.7]{DoCa92}, we conclude that  
	$$
	 \frac{\de}{\dt}(f \circ \gamma )(t) = {\cal D}f(\gamma(t)) [\dot{\gamma}(t)], \qquad t\in [0,1].
	$$
	Applying integral 
	with respect to $t$ on both sides 
	and using the Fundamental Theorem of Calculus for the real function $f \circ \gamma$, we conclude that 
	\begin{equation}\label{eq:ilemtfcfir}
		  f(p) - f(p') = f(\gamma(1)) - f( \gamma(0))   =   \int_{0}^{1} {\cal D}f(\gamma(t)) [\dot{\gamma}(t)] \dt,
	\end{equation}
	which completes the proof of \eqref{eq:tfc}.
	Now, let us prove \eqref{eq:properepsil} and \eqref{eq:Df-lipschitz}. 	
	Let $\B_{r}(\bar{p})$ be a normal ball. For each $p \in \B_{r}(\bar{p})$, the geodesic $\gamma_{\bar{p} p} \colon [0,1]  \to {\cal M}$ connecting $\bar{p}$ to $p$ can be written as $\gamma_{\bar{p}p}(t) = \exp_{\bar p}(t \exp_{\bar p}^{-1}p)$,  which implies that  $\dot{\gamma}_{\bar{p}p}(t) = P_{\bar{p} \gamma_{\bar{p}p}(t) }\exp^{-1}_{\bar p} p$ holds for all  $t\in [0,1]$. Hence,
it follows from \eqref{eq:ilemtfcfir} with $\gamma = \gamma_{\bar{p}p}$ that
	\begin{equation*}
		f(p) - f( \bar{p}) =	\int_{0}^{1} {\cal D}f(\gamma_{\bar{p}p}(t)) [ P_{\bar{p} \gamma_{\bar{p}p}(t) }\exp^{-1}_{\bar p} p ]  \dt  .
	\end{equation*}
	By adding the term $-{\cal D}f(\bar{p})[\exp^{-1}_{\bar p} p]$ on both sides, it comes that 
	$$
	f(p) - f(\bar{p}) - {\cal D}f(\bar{p})[\exp^{-1}_{\bar p} p] = \int_{0}^{1}	\left(    {\cal D}f(\gamma_{\bar{p}p}(t)) P_{\bar{p} \gamma_{\bar{p}p}(t) }   -  {\cal D}f(\bar{p}) \right) [\exp^{-1}_{\bar p}p] \, \dt .
	$$
	Applying the Euclidean norm of $\mathbb{R}^m$, we get
	\begin{align*}
		\|f(p) - f(\bar{p}) - {\cal D}f(\bar{p}) [\exp^{-1}_{\bar p} p] \|_e
		& = \left\| \int_{0}^{1}	\left(
		{\cal D}f(\gamma_{\bar{p}p}(t)) P_{\bar{p} \gamma_{\bar{p}p}(t) }   -  {\cal D}f(\bar{p}) \right)
		[\exp^{-1}_{\bar p}p] \, \dt \right\|_e\\
			& \leq  \int_{0}^{1}  \left\| 	\left(
		{\cal D}f(\gamma_{\bar{p}p}(t)) P_{\bar{p} \gamma_{\bar{p}p}(t) }   -  {\cal D}f(\bar{p}) \right)
		[\exp^{-1}_{\bar p}p] \right\|_e \, \dt \\
		& \leq \int_{0}^{1} \|	{\cal D}f(\gamma_{\bar{p}p}(t)) P_{\bar{p} \gamma_{\bar{p}p}(t) }   -  {\cal D}f(\bar{p}) \|_{map} \, \|\exp^{-1}_{\bar p} p\|_{\bar p} \,\, \dt.
	\end{align*}
It follows from \eqref{eq:dist.iq.nor.expi}
      and simple properties of integral that 
    \begin{equation}\label{eq:plaii-sup}
    	\|f(p) - f(\bar{p}) - {\cal D}f(\bar{p}) [\exp^{-1}_{\bar p} p] \|_e 
    	 \leq \sup_{t\in [0,1]}\left\{ \|	{\cal D}f( \gamma_{\bar{p}p}(t) ) P_{\bar{p} \gamma_{\bar{p}p}(t) }   -  {\cal D}f(\bar{p}) \|_{map}  \right\} \, d(p,\bar{p}), 
    \end{equation}
   for all $p \in \B_{r}(\bar{p})$. For arbitrary  $\epsilon > 0$, it follows from Lemma~\ref{lem:bounddiff} and  $\delta_{\epsilon} \in (0,r) $ that $	\| {\cal D}f(p)P_{ \bar{p}p}  -  {\cal D}f(\bar{p}) \|_{map} \leq \epsilon$ for all $ p \in \B_{\delta_{\epsilon}}(\bar{p}),$ which implies that $	\| {\cal D}f(\gamma_{\bar{p}p}(t))P_{ \bar{p}\gamma_{\bar{p}p}(t)}  -  {\cal D}f(\bar{p}) \|_{map} \leq \epsilon$ for all $ p \in \B_{\delta_{\epsilon}}(\bar{p})$ and $t\in [0,1]$.  Therefore,   \eqref{eq:plaii-sup}  yields
    \eqref{eq:properepsil}. Finally, assume that ${\cal D}f$ is $L$-Lipschitz continuous around $\bar{p}$. 
    By Definition~\ref{Def:DLips}, there exists $\delta_L \in (0, r )$ such that 
    $\|	{\cal D}f(p) P_{\bar{p}p}   -  {\cal D}f(\bar{p}) \|_{map} \leq L d(p,\bar{p})$ for all $p \in \B_{\delta_L}(\bar p)$, which implies that
   $$
   \sup_{t\in [0,1]}\left\{ \|	{\cal D}f(\gamma_{\bar{p}p}(t)) P_{\bar{p}\gamma_{\bar{p}p}(t)}   -  {\cal D}f(\bar{p}) \|_{map} \right\} \leq L \sup_{t\in [0,1]}\left\{ d(\gamma_{\bar{p}p}(t),\bar{p}) \right\}= L d(p, \bar{p}),
   $$
    for all $p \in \B_{\delta_L}(\bar p)$.	Using the previous inequality and \eqref{eq:plaii-sup}, we conclude the proof of \eqref{eq:Df-lipschitz}.	
\end{proof}

The following proposition provides a geometric property of the exponential map. 
\begin{proposition}\label{lem:geo.mpoit.exp}
	Let $\B_{r}(\bar{p})$ be a  totally normal ball. Then, for each $p$ and $q$ in $ \B_{r}(\bar{p})$, we have
	\begin{equation}\label{eq:lem.ap.geo.d.exp}
		\dot{\gamma}_{pq}(t) = \exp_{\gamma_{pq}(t)}^{-1}q -  \exp_{\gamma_{pq}(t)}^{-1}p, \qquad  t\in [0,1],
	\end{equation}
	where $\gamma_{pq} \colon [0,1] \to {\cal M}$ is the geodesic connecting $p=\gamma_{pq}(0)$ to $q=\gamma_{pq}(1)$.
\end{proposition}
\begin{proof}
	Let  $p$ and $q$ be arbitrary points in $\B_{r}(\bar{p})$. Note that \eqref{eq:lem.ap.geo.d.exp} is trivially satisfied when $t=0$ and $t=1$. Thus, we only need to analyze the case where $t\in (0,1)$. Considering that the geodesic starting from
	 $p$ with direction $\dot{\gamma}_{pq}(0)$ is unique, and that 
	  $ \B_{r}(\bar{p})$ is a  totally normal ball, we have
	$$
	\frac{1}{d(p,q)}    \dot{\gamma}_{pq}(0) = \frac{1}{d(p, \gamma_{pq}(t))}  \exp_p^{-1}(\gamma_{pq}(t)), \quad t\in (0,1).
	$$
	Applying parallel transport on both sides of the equation, we obtain
	$$
	\frac{1}{d(p,q)}    P_{p\gamma_{pq}(t)} [\dot{\gamma}_{pq}(0)] = \frac{1}{d(p, \gamma_{pq}(t))}  P_{p\gamma_{pq}(t)} [ \exp_p^{-1}(\gamma_{pq}(t))]= \frac{1}{d(p, \gamma_{pq}(t))}   [- \exp_{\gamma_{pq}(t)}^{-1}p], \quad t\in (0,1).
	$$
	Since $\gamma_{pq}$ is a geodesic, the field $\dot{\gamma}_{pq}$ is  parallel  along $\gamma_{pq}$, and, consequently,  the equality $	\dot{\gamma}_{pq}(t) = P_{p\gamma_{pq}(t)}  [\dot{\gamma}_{pq}(0)]$ holds for all $t\in (0,1)$. Thus, the last expression implies
	\begin{equation*}\label{eq:vmdinlemaappor}
		\frac{d(p, \gamma_{pq}(t))}{d(p,q)}      \dot{\gamma}_{pq}(t)  =   -  \exp_{\gamma_{pq}(t)}^{-1}(p), \quad t\in (0,1).
	\end{equation*}
	In a similar manner, we can show that
	$$
	\frac{d( \gamma_{pq}(t) , q )}{d(p,q)} \dot{\gamma}_{pq}(t) =      \exp_{\gamma_{pq}(t)}^{-1}(q), \quad t\in (0,1).
	$$
	To conclude the proof, simply add the last two equalities and use the fact that $d(p, \gamma_{pq}(t)) + d(\gamma_{pq}(t), q) = d(p, q)$ for all $t \in (0, 1)$.
\end{proof}
In the next result, we will establish a sufficient condition for the function $f: \mathcal{M} \to \mathbb{R}$ to have closed graph.
\begin{proposition}\label{prop:gph.closed}
	Let $f\colon {\cal M} \to \mathbb{R}$ be a continuous function. Then, the graph of $f$ is closed in ${\cal M} \times \mathbb{R}$.
\end{proposition}
\begin{proof}
	Since the function $f$ is continuous, the function $\varphi  \colon {\cal M} \times \mathbb{R} \to \mathbb{R}$ defined by $\varphi(p,x) = |f(p) - x|$ is also continuous. On the other hand, we have that 
	$$
	\gph f \coloneqq \{(p,x)\in {\cal M} \times \mathbb{R} \colon x = f(p)  \} = \varphi^{-1}(0).
	$$
	Therefore, $\gph f$ is closed because it is the inverse image, by the continuous function $\varphi$, of the closed set $\{0\}$.
\end{proof}

Throughout this paper, our strategy to prove some results requires the following lemma.
\begin{lemma}\label{lem:assumptionproof}
	Let $f \colon {\cal M} \to \mathbb{R}^m$ be a continuously differentiable function at $\bar{p}$. Then, for each $\epsilon>0$, there exists a totally normal ball	$\B_{\delta_{\epsilon}}(\bar{p})$ such that 
	\begin{equation*}\label{eq:assump:lipsh.cond}
		\| {\cal D}f(p) [\exp_{p}^{-1}p' - \exp_{p}^{-1}p''] -  {\cal D}f(\bar{p})[ \exp_{\bar{p}}^{-1}p' - \exp_{\bar{p}}^{-1}p'']\|_e
		\leq \epsilon d(p',p''), \qquad p, p',p''\in \B_{\delta_{\epsilon}}(\bar{p}).
	\end{equation*}
\end{lemma}
\begin{proof}
	Let $\B_{r}(\bar{p})$ be a totally normal ball. Define the function $\Psi \colon \B_{r}(\bar{p}) \times \B_{r}(\bar{p}) \to \mathbb{R}^m$  by
	\begin{equation}\label{eq:defpsi}
	\Psi(p,q) =   {\cal D}f(p) [\exp_{p}^{-1}q]  -  {\cal D}f(\bar p) [\exp_{\bar p}^{-1}q].
	\end{equation}
	Let $\epsilon > 0$, and consider $\cal M \times \cal M$ with the product metric.
	Since ${\cal D}f$ and $\exp^{-1}$ are continuous at $\bar{p}$ and continuously differentiable at $(\bar{p}, \bar{p})$, respectively, ${\cal D}_2\Psi(p,q)$ (the differential of $\Psi(p,\cdot)$ at $q$) is continuous at $(\bar{p}, \bar{p})$. Therefore, there exists $\delta \in (0, r)$ such that
    $$
    d((p,q),(\bar{p},\bar{p}))=(d^2(p,\bar p) + d^2(q,\bar p))^{\frac{1}{2}}< \delta \quad \mbox{implies} \quad \| {\cal D}_2\Psi(p,q) - {\cal D}_2 \Psi(\bar p,\bar p) \|_e< \epsilon.
    $$
    From \eqref{eq:defpsi}, it follows that $\Psi(\bar p, q)=0$ for all $q\in \B_{r}(\bar{p})$ which implies  ${\cal D}_2 \Psi(\bar p,\bar p)=0$. Thus, the last line can be rewritten as
    \begin{equation}\label{eq:difmenepi}
    d^2(p,\bar p) + d^2(q,\bar p) < \delta^2 \quad \mbox{implies} \quad \| {\cal D}_2\Psi(p,q) \|_e< \epsilon.
    \end{equation}
    Set $ \delta_{\epsilon} \coloneqq \delta/\sqrt{2}$. It follows from $d^2(p,\bar p) + d^2(q,\bar p) < \delta^2$  for all $(p,q)\in  \B_{\delta_{\epsilon}}(\bar{p}) \times \B_{\delta_{\epsilon}}(\bar{p})$ and  \eqref{eq:difmenepi} that
$\| {\cal D}_2\Psi(p,q) \|_e< \epsilon$ for all $(p,q)\in  \B_{\delta_{\epsilon}}(\bar{p}) \times \B_{\delta_{\epsilon}}(\bar{p})$. Consequently,
     \begin{equation}\label{eq:boundD2}
     \sup_{q\in {\cal B}_{\delta_{\epsilon}}(\bar{p})}\{	\| {\cal D}_2\Psi(p,q) \|_e \}\leq \epsilon \quad \mbox{for all } p\in \B_{\delta_{\epsilon}}(\bar{p}).
     	\end{equation}
For  $p,p',p''\in \B_{\delta_{\epsilon}}(\bar{p})$, it follows from the first part of Proposition~\ref{lem:properepsil}  and \eqref{eq:defpsi} that
     \begin{align*}
     {\cal D}f(p) [\exp_{p}^{-1}p' - \exp_{p}^{-1}p''] -  {\cal D}f(\bar{p})[ \exp_{\bar{p}}^{-1}p' - \exp_{\bar{p}}^{-1}p'']
     &= \Psi(p,p') - \Psi(p,p''), \\
     &= \int_{0}^{1}  {\cal D}_2\Psi(p, \gamma(t) ) [\dot{\gamma}(t)] \dt,
     \end{align*}
where $\gamma \colon [0,1] \to {\cal M}$ is a geodesic that satisfies $\gamma(0)=p''$ and $\gamma(1)=p'$. Since $ \B_{\delta_{\epsilon}}(\bar{p}) $ is totally normal, $\gamma(t) \in \B_{\delta_{\epsilon}}(\bar{p})$ for all $t\in[0,1]$. Hence, applying the norm to both sides of the above equality and using norm properties, we obtain
	\begin{align*}
	\| {\cal D}f(p) [\exp_{p}^{-1}p' - \exp_{p}^{-1}p''] -  {\cal D}f(\bar{p})[ \exp_{\bar{p}}^{-1}p' - \exp_{\bar{p}}^{-1}p''] \|_e
	&\leq \int_{0}^{1} \| {\cal D}_2\Psi(p, \gamma(t) ) [\dot{\gamma}(t)] \|_e \dt, \\
	 &\leq \int_{0}^{1}  \| {\cal D}_2\Psi(p, \gamma(t) ) \|_{map} \| \dot{\gamma}(t) \|_{\gamma(t)} \dt, \\
	 & \leq \sup_{q \in {\cal B}_{\delta_{\epsilon}}(\bar{p}) } \left\{  \| {\cal D}_2\Psi(p, q ) \|_{map} \right\} d(p',p'').
	\end{align*}
	The conclusion of this proof follows from the previous inequality and \eqref{eq:boundD2}.
\end{proof}


\def\cprime{$'$}


\begin{thebibliography}{10}

\bibitem{absil2009optimization}
P.-A. Absil, R.~Mahony, and R.~Sepulchre.
\newblock {\em Optimization algorithms on matrix manifolds}.
\newblock Princeton University Press, 2009.

\bibitem{adler2002newton}
R.~L. Adler, J.-P. Dedieu, J.~Y. Margulies, M.~Martens, and M.~Shub.
\newblock Newton's method on {R}iemannian manifolds and a geometric model for
  the human spine.
\newblock {\em IMA Journal of Numerical Analysis}, 22(3):359--390, 2002.

\bibitem{Adly2015}
S.~Adly, R.~Cibulka, and H.~van Ngai.
\newblock Newton's method for solving inclusions using set-valued
  approximations.
\newblock {\em SIAM J. Optim.}, 25:159--184, 2015.

\bibitem{Adly2018}
S.~Adly, H.~van Ngai, and N.~V. Vu.
\newblock Newton-type method for solving generalized equations on {R}iemannian
  manifolds.
\newblock {\em Journal of Convex Analysis}, 25(2):341--370, 2018.

\bibitem{Adly2022}
S.~Adly, H.~van Ngai, and N.~V. Vu.
\newblock {D}ennis-{M}or\'e condition for set-valued vector fields and the
  superlinear convergence of {B}royden updates in {R}iemannian manifolds.
\newblock {\em Journal of Convex Analysis}, 29(3), 2022.

\bibitem{alvarez2008unifying}
F.~Alvarez, J.~Bolte, and J.~Munier.
\newblock A unifying local convergence result for {N}ewton's method in
  {R}iemannian manifolds.
\newblock {\em Foundations of Computational Mathematics}, 8(2):197--226, 2008.

\bibitem{andreani2022}
R.~Andreani, R.~M. de~Carvalho, L.~D. Secchin, and G.~N. Silva.
\newblock Convergence of {Q}uasi-{N}ewton methods for solving constrained
  generalized equations.
\newblock {\em ESAIM:COCV}, 28, 2022.

\bibitem{DontchevAragon2014}
F.~J. Arag{\'o}n~Artacho, A.~Belyakov, A.~L. Dontchev, and M.~L{\'o}pez.
\newblock Local convergence of quasi-{N}ewton methods under metric regularity.
\newblock {\em Comput. Optim. Appl.}, 58(1):225--247, 2014.

\bibitem{DontchevAragon2011}
F.~J. Arag{\'o}n~Artacho, A.~L. Dontchev, M.~Gaydu, M.~H. Geoffroy, and V.~M.
  Veliov.
\newblock Metric regularity of {N}ewton's iteration.
\newblock {\em SIAM J. Control Optim.}, 49(2):339--362, 2011.

\bibitem{BergmannHerzog2019}
R.~Bergmann and R.~Herzog.
\newblock Intrinsic formulation of {KKT} conditions and constraint
  qualifications on smooth manifolds.
\newblock {\em SIAM Journal on Optimization}, 29(4):2423--2444, 2019.

\bibitem{bergmann2021fenchel}
R.~Bergmann, R.~Herzog, M.~Silva~Louzeiro, D.~Tenbrinck, and
  J.~Vidal-N{\'u}{\~n}ez.
\newblock Fenchel duality theory and a primal-dual algorithm on {R}iemannian
  manifolds.
\newblock {\em Foundations of Computational Mathematics}, 21(6):1465--1504,
  2021.

\bibitem{bergmann2016parallel}
R.~Bergmann, J.~Persch, and G.~Steidl.
\newblock A parallel {D}ouglas--{R}achford algorithm for minimizing rof-like
  functionals on images with values in symmetric {H}adamard manifolds.
\newblock {\em SIAM Journal on Imaging Sciences}, 9(3):901--937, 2016.

\bibitem{bini2013computing}
D.~A. Bini and B.~Iannazzo.
\newblock Computing the karcher mean of symmetric positive definite matrices.
\newblock {\em Linear Algebra and its Applications}, 438(4):1700--1710, 2013.

\bibitem{boumal2023intromanifolds}
N.~Boumal.
\newblock {\em An introduction to optimization on smooth manifolds}.
\newblock Cambridge University Press, 2023.

\bibitem{Censor2011}
Y.~Censor, A.~Gibali, and S.~Reich.
\newblock Algorithms for the split variational inequality problem.
\newblock {\em Numerical Algorithms}, 59:301--323, 2011.

\bibitem{Dontchev2015}
R.~Cibulka, A.~Dontchev, and M.~H. Geoffroy.
\newblock Inexact {N}ewton {M}ethods and {D}ennis--{M}or\'e {T}heorems for
  {N}onsmooth {G}eneralized {E}quations.
\newblock {\em SIAM J. Control Optim.}, 53(2):1003--1019, 2015.

\bibitem{CibulkaDontchev2015}
R.~Cibulka, A.~Dontchev, and M.~H. Geoffroy.
\newblock Inexact {N}ewton methods and {D}ennis--{M}oré theorems for nonsmooth
  generalized equations.
\newblock {\em SIAM Journal on Control and Optimization}, 53(2):1003--1019,
  2015.

\bibitem{Oliveira2019}
F.~R. de~Oliveira, O.~P. Ferreira, and G.~N. Silva.
\newblock Newton’s method with feasible inexact projections for solving
  constrained generalized equations.
\newblock {\em Computational Optimization and Applications}, 72:159--177, 2019.

\bibitem{dedieu2003newton}
J.-P. Dedieu, P.~Priouret, and G.~Malajovich.
\newblock Newton's method on {R}iemannian manifolds: covariant alpha theory.
\newblock {\em IMA Journal of Numerical Analysis}, 23(3):395--419, 2003.

\bibitem{DoCa92}
M.~P. do~Carmo.
\newblock {\em Riemannian geometry}.
\newblock Mathematics: Theory \& Applications. Birkh\"auser Boston, Inc.,
  Boston, MA, 1992.
\newblock Translated from the second Portuguese edition by Francis Flaherty.

\bibitem{Dontchev1996}
A.~L. Dontchev.
\newblock Local analysis of a {N}ewton-type method based on partial
  linearization.
\newblock In {\em The mathematics of numerical analysis ({P}ark {C}ity, {UT},
  1995)}, volume~32 of {\em Lectures in Appl. Math.}, pages 295--306. Amer.
  Math. Soc., Providence, RI, 1996.

\bibitem{DontchevRockafellar2010}
A.~L. Dontchev and R.~T. Rockafellar.
\newblock Newton's method for generalized equations: a sequential implicit
  function theorem.
\newblock {\em Math. Program.}, 123(1, Ser. B):139--159, 2010.

\bibitem{DontchevRockafellar2013}
A.~L. Dontchev and R.~T. Rockafellar.
\newblock Convergence of inexact {N}ewton methods for generalized equations.
\newblock {\em Math. Program.}, 139(1-2, Ser. B):115--137, 2013.

\bibitem{DontchevRockafellar2014}
A.~L. Dontchev and R.~T. Rockafellar.
\newblock {\em Implicit Functions and Solution Mappings: A View from
  Variational Analysis}.
\newblock Springer, New York, 2nd edition, 2014.

\bibitem{fernandes2017superlinear}
T.~A. Fernandes, O.~P. Ferreira, and J.~Yuan.
\newblock On the superlinear convergence of {N}ewton’s method on {R}iemannian
  manifolds.
\newblock {\em Journal of Optimization Theory and Applications},
  173(3):828--843, 2017.

\bibitem{FerreiraSilva2018}
O.~Ferreira and G.~Silva.
\newblock Local convergence analysis of {N}ewton's method for solving strongly
  regular generalized equations.
\newblock {\em Journal of Mathematical Analysis and Applications},
  458(1):481--496, 2018.

\bibitem{ferreira2013projections}
O.~P. Ferreira, A.~N. Iusem, and S.~Z. N{\'e}meth.
\newblock Projections onto convex sets on the sphere.
\newblock {\em Journal of Global Optimization}, 57(3):663--676, 2013.

\bibitem{ferreira2017metrically}
O.~P. Ferreira, C.~Jean-Alexis, and A.~Pi{\'e}trus.
\newblock Metrically regular vector field and iterative processes for
  generalized equations in {H}adamard manifolds.
\newblock {\em Journal of Optimization Theory and Applications},
  175(3):624--651, 2017.

\bibitem{ferreira2019gradient}
O.~P. Ferreira, M.~S. Louzeiro, and L.~Prudente.
\newblock Gradient method for optimization on riemannian manifolds with lower
  bounded curvature.
\newblock {\em SIAM Journal on Optimization}, 29(4):2517--2541, 2019.

\bibitem{FerreiraSilva2017}
O.~P. Ferreira and G.~N. Silva.
\newblock Kantorovich's {T}heorem on {N}ewton's {M}ethod for {S}olving
  {S}trongly {R}egular {G}eneralized {E}quation.
\newblock {\em SIAM J. Optim.}, 27(2):910--926, 2017.

\bibitem{GayduSilva2020}
M.~Gaydu and G.~N. Silva.
\newblock A general iterative procedure to solve generalized equations with
  differentiable multifunction.
\newblock {\em Journal of Optimization Theory and Applications}, 185:207--222,
  2020.

\bibitem{genicot2015weakly}
M.~Genicot, W.~Huang, and N.~T. Trendafilov.
\newblock Weakly correlated sparse components with nearly orthonormal loadings.
\newblock In {\em International Conference on Geometric Science of
  Information}, pages 484--490. Springer, 2015.

\bibitem{GEOFFROY2004}
M.~H. Geoffroy and A.~Piétrus.
\newblock Local convergence of some iterative methods for generalized
  equations.
\newblock {\em Journal of Mathematical Analysis and Applications},
  290(2):497--505, 2004.

\bibitem{He2015}
H.~He, C.~Ling, and H.-K. Xu.
\newblock A relaxed projection method for split variational inequalities.
\newblock {\em Journal of Optimization Theory and Applications}, 166:213--233,
  2015.

\bibitem{hesselholtvector}
L.~Hesselholt.
\newblock Vector fields on spheres.
\newblock {\em Preprint, http://www-math. mit. edu~
  larsh/teaching/vectorfields. pdf}.

\bibitem{izmailov2010inexact}
A.~F. Izmailov and M.~V. Solodov.
\newblock Inexact {J}osephy--{N}ewton framework for generalized equations and
  its applications to local analysis of {N}ewtonian methods for constrained
  optimization.
\newblock {\em Computational Optimization and Applications}, 46(2):347--368,
  2010.

\bibitem{karcher1977riemannian}
H.~Karcher.
\newblock Riemannian center of mass and mollifier smoothing.
\newblock {\em Communications on pure and applied mathematics}, 30(5):509--541,
  1977.

\bibitem{KlatteKummer}
D.~Klatte and B.~Kummer.
\newblock Approximations and generalized {N}ewton methods.
\newblock {\em Mathematical Programming}, 168:673--716, 2018.

\bibitem{Lang1999}
S.~Lang.
\newblock {\em Fundamentals of differential geometry}, volume 191 of {\em
  Graduate Texts in Mathematics}.
\newblock Springer-Verlag, New York, 1999.

\bibitem{Lee:2003:1}
J.~M. Lee.
\newblock {\em Introduction to Smooth Manifolds}, volume 218 of {\em Graduate
  Texts in Mathematics}.
\newblock Springer-Verlag, New York, 2003.

\bibitem{lee2006riemannian}
J.~M. Lee.
\newblock {\em Riemannian manifolds: an introduction to curvature}, volume 176.
\newblock Springer Science \& Business Media, 2006.

\bibitem{li2005convergence}
C.~Li and J.~Wang.
\newblock Convergence of the {N}ewton method and uniqueness of zeros of vector
  fields on {R}iemannian manifolds.
\newblock {\em Science in China Series A: Mathematics}, 48(11):1465--1478,
  2005.

\bibitem{LiChong-YaoJenChih.2012}
C.~Li and J.-C. Yao.
\newblock Variational inequalities for set-valued vector fields on {R}iemannian
  manifolds: Convexity of the solution set and the proximal point algorithm.
\newblock {\em SIAM Journal on Control and Optimization}, 50(4):2486--2514,
  2012.

\bibitem{ShuLongLi.ChongLi.YeongChengLiou.JenChihYao.2009}
S.-L. Li, C.~Li, Y.-C. Liou, and J.-C. Yao.
\newblock Existence of solutions for variational inequalities on {R}iemannian
  manifolds.
\newblock {\em Nonlinear Analysis: Theory, Methods \& Applications},
  71(11):5695--5706, 2009.

\bibitem{liu2020simple}
C.~Liu and N.~Boumal.
\newblock Simple algorithms for optimization on {R}iemannian manifolds with
  constraints.
\newblock {\em Applied Mathematics \& Optimization}, 82(3):949--981, 2020.

\bibitem{nemeth2003}
S.~N\'emeth.
\newblock Variational inequalities on hadamard manifolds.
\newblock {\em Nonlinear Analysis: Theory, Methods \& Applications},
  52(5):1491--1498, 2003.

\bibitem{sato2021riemannian}
H.~Sato.
\newblock {\em Riemannian optimization and its applications}.
\newblock Springer, 2021.

\bibitem{SraHosseini2015}
S.~Sra and R.~Hosseini.
\newblock Conic geometric optimization on the manifold of positive definite
  matrices.
\newblock {\em SIAM J. Optim.}, 25(1):713--739, 2015.

\bibitem{tang2012unsupervised}
J.~Tang and H.~Liu.
\newblock Unsupervised feature selection for linked social media data.
\newblock In {\em Proceedings of the 18th ACM SIGKDD international conference
  on Knowledge discovery and data mining}, pages 904--912, 2012.

\bibitem{Tu:2011:1}
L.~W. Tu.
\newblock {\em An Introduction to Manifolds}.
\newblock Universitext. Springer, New York, 2 edition, 2011.

\bibitem{Udriste1994}
C.~Udri{\c{s}}te.
\newblock {\em Convex functions and optimization methods on {R}iemannian
  manifolds}, volume 297 of {\em Mathematics and its Applications}.
\newblock Kluwer Academic Publishers Group, Dordrecht, 1994.

\bibitem{ChongLi2009}
J.-H. Wang, S.~Huang, and C.~Li.
\newblock Extended {N}ewton's method for mappings on {R}iemannian manifolds
  with values in a cone.
\newblock {\em Taiwanese J. Math.}, 13(2B):633--656, 2009.

\bibitem{weber2023riemannian}
M.~Weber and S.~Sra.
\newblock Riemannian optimization via frank-wolfe methods.
\newblock {\em Mathematical Programming}, 199(1):525--556, 2023.

\bibitem{zhang2017global}
Y.~Zhang, Y.~Lau, H.-w. Kuo, S.~Cheung, A.~Pasupathy, and J.~Wright.
\newblock On the global geometry of sphere-constrained sparse blind
  deconvolution.
\newblock In {\em Proceedings of the IEEE Conference on Computer Vision and
  Pattern Recognition}, pages 4894--4902, 2017.

\end{thebibliography}
\end{document}